\documentclass[10pt,leqno,amscd,amssymb,verbatim, url]{amsart}
\usepackage{amsfonts,amssymb,graphicx}
\usepackage{amsmath,amscd}
\newtheorem{thm}{Theorem}[section]
\usepackage{amsfonts,latexsym, curves}
\newtheorem{lem}[thm]{Lemma}
\newtheorem{cor}[thm]{Corollary}

\theoremstyle{definition}

\newtheorem{ques}[thm]{Question}
\newtheorem{rem}[thm]{Remark}
\newcommand{\bpict}{\begin{picture}}

\newcommand{\epict}{\end{picture}}

\newtheorem{rems}[thm]{Remarks}

\numberwithin{equation}{thm}

\newcommand{\N}{{\mathbb N}}

\newcommand{\Z}{{\mathbb Z}}

\newcommand{\Ext}{{\text{\rm Ext}}}

\newcommand{\Hom}{\text{\rm Hom}}

\newcommand{\ch}{\operatorname{ch}}
\newcommand{\ind}{\operatorname{ind}}

\newcommand{\res}{\operatorname{res}}

\newcommand{\Ker}{\operatorname{Ker}}

\newcommand{\blist}{\begin{list}{\rom{(\roman{enumi})}}{\setlength
{\leftmarg in}{0em} \setlength{\itemindent}{7ex}
\setlength{\labelsep}{2ex}\setlength{\listparindent}{\parindent}
\usecounter{enumi}}}
\newcommand{\elist}{\end{list}}

\newcommand{\opH}{\text{\rm H}}
\newcommand{\gen}{\text{\rm gen}}
\newcommand{\Ind}{\text{\rm Ind}}
\newcommand{\Map}{\text{\rm Map}}
\newcommand{\mF}{{\mathfrak F}}
\newcommand{\sG}{{\mathcal G}}

\begin{document}
\begin{abstract}

The idea that the cohomology of finite groups might be fruitfully approached via the cohomology 
of  ambient semisimple algebraic groups was first shown to be viable in the   papers \cite{CPS75} and 
\cite{CPSK77}. The second paper introduced, through a limiting process, the notion of generic cohomology, as an intermediary between finite Chevalley group and algebraic group cohomology. 

The present paper shows that, for irreducible modules as coefficients, the limits can be eliminated in all but finitely many cases.
These exceptional cases depend
only on the root system and cohomological degree. In fact, we show that, for sufficiently large $r$, depending only on
the root system and $m$, and not on the prime $p$ or the irreducible module $L$, there are isomorphisms
$\opH^m(G(p^r),L)\cong\opH^m(G(p^r), L')\cong \opH^m_\gen(G,L')\cong \opH^m(G,L')$, where
the subscript ``gen" refers to generic cohomology and $L'$ is a constructibly determined irreducible ``shift" of the (arbitrary) irreducible
module $L$ for the finite Chevalley group $G(p^r)$. By a famous theorem of Steinberg, both $L$ and $L'$ extend to irreducible modules
for the ambient algebraic group $G$ with $p^r$-restricted highest weights. This leads to the notion of a module or weight being ``shifted $m$-generic,"
and thus to the title of this paper.  Our approach is based on questions raised by the third author in \cite{SteSL3}, which we answer here in the cohomology cases. We obtain many additional results, often with formulations in the more general context of $\Ext^m_{G(q)}$
with irreducible coefficients.
 \end{abstract}
 
 \title[Shifted generic cohomology]{Shifted generic cohomology}
\author{Brian J. Parshall}
\address{Department of Mathematics \\
University of Virginia\\
Charlottesville, VA 22903} \email{bjp8w@virginia.edu {\text{\rm
(Parshall)}}}
\author{Leonard L. Scott}
\address{Department of Mathematics \\
University of Virginia\\
Charlottesville, VA 22903} \email{lls2l@virginia.edu {\text{\rm
(Scott)}}}
\author{David I. Stewart}
\address{New College, Oxford\\ Oxford, UK} \email{david.stewart@new.ox.ac.uk {\text{\rm(Stewart)}}}

\thanks{Research supported in part by the National Science
Foundation}
\medskip \maketitle
\section{Introduction}

Let $G$ be a simply connected, semisimple  algebraic group defined and split over the prime
field ${\mathbb F}_p$ of positive characteristic $p$. Write $k=\bar{\mathbb F}_p$. For a power $q=p^r$, let $G(q)$ be the
subgroup of ${\mathbb F}_q$-rational points in $G$. Thus, $G(q)$ is a finite Chevalley
group. Let $M$ be a finite dimensional rational $G$-module and let $m$ be a non-negative integer. In \cite{CPSK77}, the
first two authors of this paper, together with Ed Cline and Wilberd van der Kallen, defined the
notion of the generic $m$-cohomology
$$\opH^m_{\text{\rm gen}}(G,M):=\underset{\underset{q}\leftarrow}\lim\,\opH^m(G(q),M)$$
of $M$.
The limit is, in fact, a {\it stable} limit for any given $M$. Moreover, $\opH^m_\gen(G,M)\cong \opH^m(G, M^{[e_0]})$, where $M^{[e_0]}$ denotes the twist of $M$ through some $e_0$th power of the Frobenius endomorphism of $G$. Although the non-negative integer $e_0$  may be chosen independently of $p$ and $M$, it can also be chosen as a function $e_0(M)$ of $M$.
Unfortunately, given a rational $G$-module $M$ for which one wants to compute $\opH^m(G(q),M)$, it is
frequently necessary to take $e_0(M)>0$. This problem has been noted by others \cite[\S1]{Georgia}. Worse, it may be necessary to enlarge $q$ in order to obtain
$\opH^m(G(q),M)\cong \opH_\gen^m(G,M)$. The problem is exacerbated if one is interested in calculations
for an infinite family of modules $M$, such as the irreducible $G$-modules. By a famous result of Steinberg,
all irreducible $kG(q)$-modules are, up to isomorphism, the restrictions to $G(q)$ of the irreducible rational $G$-modules whose highest weights are $q$-restricted. 

We propose here a remedy to this situation. Observe that, for any $q$-restricted
dominant weight $\lambda$ and non-negative integer $e$, there is a unique $q$-restricted dominant weight $\lambda'$ with $L(\lambda)^{[e]}|_{G(q)}
=L(\lambda')|_{G(q)}$. Write $\lambda'=\lambda^{[e]_q}$ and $L(\lambda')=L(\lambda)^{[e]_q}$.
We shall refer to any weight $\lambda'$ of this form as a {\it $q$-shift of }$\lambda$. The 
main result, Theorem \ref{shiftedGeneric}, in this paper shows  that, for $r\gg 0$\footnote{The lower bound on $r$ here depends only on the root system $\Phi$ of $G$ and
the cohomological degree $m$, and not on $p$ or $\lambda$. Moreover, this bound can be recursively determined.}, and 
any $q$-restricted dominant weight $\lambda$  
\begin{equation}\label{main}\opH^m(G(q), L(\lambda))\cong\opH^m_\gen(G, L(\lambda'))\cong \opH^m(G,L(\lambda')),\end{equation}
for some $q$-restricted weight $\lambda'=\lambda^{[e]_q}$ with $e=e(\lambda)=e(\lambda,q)\geq 0$.   Similar results hold for $\Ext^m_{G(q)}(L(\mu),L(\lambda))$ with $\lambda,\mu$ both
$q$-restricted, though with some conditions on $\mu$. The first isomorphism in (\ref{main}) may be viewed
as saying that $L(\lambda)$ is ``shifted $m$-generic at $q$"; see the end of this introduction.\footnote
{We ask in Question \ref{conjecture} below if the $\Ext$-analog of the first isomorphism holds for all $q$-restricted $\lambda$ and
$\mu$, though we know conditions on $\mu$ are needed for the second.}  The map $\lambda\mapsto\lambda^{[e]_q}$ defines an action of the cyclic group ${\mathbb Z}/r{\mathbb Z}$ on the set $X^+_r$ of $q$-restricted
weights, and $\lambda'$ in (\ref{main}) is a ``distinguished" member in the orbit of $\lambda$ under
this action, chosen to optimize the positions of zero terms in its $p$-adic expansion. 

  The origin of these results goes
back to $\Ext^m$-questions raised by the third author in \cite[\S3]{SteSL3}, where the $q$-shift
$\lambda^{[e]_q}$ of $\lambda$ was denoted $\lambda^{\{e\}}$, and $L(\lambda^{\{e\}})$ was called a $q$-wrap of $L(\lambda)$.   While raised for general $m\geq 0$, these questions arose in part  from observations for $m=1,2$, namely, from noting 
a parallel between the $2$-cohomology result \cite[Thm. 2]{SteSL3} and a $1$-cohomology result in \cite[Thm. 5.5]{BNP06}, which also had an $\Ext^1$-analog \cite[Thm. 5.6]{BNP06}---the conclusions of all these
results involve what we now call $q$-shifted weights in their formulation.\footnote{As far as we know, the
\cite[Thm. 5.6]{BNP06} is the first use of $q$-shifted weights in a general homological theorem. However, this shifting (or wrapping) behavior for $SL_2$ had been observed much earlier: see \cite[Cor. 4.5]{AJL83}.} Essentially,
our main Theorem \ref{shiftedGeneric} provides a strong answer to \cite[Question 3.8]{SteSL3} in the cohomology cases, in 
addition to interpreting it in terms of generic cohomology.  Also,
Theorem \ref{thm6.2}(c) proves a similar result for $\Ext^m_{G(q)}(L(\mu),L(\lambda))$ when $p$ is sufficiently large,
and with no requirement on $r$, but with $\lambda$ and $\mu$ required to have a zero digit in common (i.~e.,
$\lambda_i=\mu_i=0$ for some $i<r$, using the terminology below).
Remark \ref{needATwist} gives an example showing this result is near best possible, especially when $\lambda=\mu$, and that the original  \cite[Question 3.8]{SteSL3} must be reformulated. Such a reformulation is given
in Question \ref{conjecture}.

Our investigation yields many other useful results.
  We
mention a few. First, any dominant weight $\lambda$ has a $p$-adic expansion $\lambda=\lambda_0+p\lambda_1 + p^2\lambda_2
+\cdots$, where each $\lambda_i$ is $p$-restricted. We call the pairs $(i,\lambda_i)$ \emph{digits} of
$\lambda$, and we say a digit is 0 if $\lambda_i=0$. Theorem \ref{digitsForG(q)} states that, given
$m\geq 0$, there is a constant $d$, depending only on $\Phi$ and $m$,  such that, for any prime $p$,  any power $q=p^r$, and any pair of $q$-restricted weights $\lambda,\mu$, if $\Ext^m_{G(q)}(L(\mu),L(\lambda))\not=0$, then $\lambda$ and $\mu$ differ by at most  $d$ digits. Thus, in the cohomology case,  
 if $\opH^m(G(q), L(\lambda))\not=0$, then $\lambda$ has at most $d$ nonzero
digits. Versions of these results hold for both rational $G$-cohomology and $\Ext^m$-groups; see Theorem
\ref{digits}. These 
digit bounding results were inspired by \cite[Question 3.10]{SteSL3}, which we answer completely.

Second, combining the main Theorem \ref{shiftedGeneric} with the large prime cohomology results \cite[Thm. 7.5]{BNP01} gives a new
 proof\footnote{The first proof that such a bound exists is unpublished, part of joint work in progress on homological bounds for finite groups of Lie type by the three authors of this paper
 together with C. Bendel, D. Nakano, and C. Pillen. } that there is a bound on $\dim\opH^m(G(q),L(\lambda))$, for $q$-restricted $\lambda$, depending only on $\Phi$ and $m$, and not on $p$ or $r$. In fact, after throwing
 away finitely many values of $q$, Theorem \ref{summary} shows that  $\dim\opH^m(G(q),L(\lambda))$  is bounded by the maximum dimension of the spaces $\opH^m(G,L(\mu))$, with
 $p$ and $\mu\in X^+$ allowed to vary (with only $m$ and $\Phi$ fixed).  The latter maximum has been shown to be finite in \cite[Thm. 7.1]{PS11}. Indeed, apart from finitely many exceptional $q$, the finite group cohomology $\opH^m(G(q),L(\lambda))$ identifies
 with a rational cohomology group $\opH^m(G,L(\mu))$, for an explicitly determined dominant weight $\mu$ (which depends on $\lambda$).

Though the main focus of this paper is on results which hold for all primes $p$,  we collect several
results in Section 6, most formulated in the $\Ext^m_{G(q)}$-context, which are valid in the special case when $p$ is modestly large. One such result is Theorem \ref{thm6.2}(c) discussed above. This theorem, given in
a ``shifted generic" framework, leads also to a fairly definitive treatment of generic cohomology for large
primes in Theorem \ref{lastcor} and the Appendix.

A key ingredient in this work is the
elegant filtration, due to Bendel-Nakano-Pillen, of the induced module ${\mathcal G}_q(k)
:=\ind_{G(q)}^Gk$; see \cite{BNP11} and the other references at the start of Section 4.  This result is, in our
view, the centerpiece of a large collection of results and ideas of these authors, focused on using the induction functor $\ind_{G(q)}^G$ in concert with truncation to smaller categories of rational $G$-modules.  
The filtration of ${\mathcal G}_q(k)$ is described
in Theorem \ref{filtration} below, and we derive some consequences of it in Section 4.  

Also, the specific theorems and ideas establishing generic cohomology, as originally formulated in \cite{CPSK77}, play an important role in Section 5, both directly and as a background motivation for exploring digit bounding. 

Finally, to explain the title of this paper, a finite dimensional, rational $G$-module $M$ may be called ``$m$-generic at $q$"
if $\opH^m(G(q),M)\cong \opH^m_\gen(G,M)$.\footnote{In practice, this often happens when the stable limit defining the generic cohomology has already occurred at $q$; however, we do not make this part of the
definition.}   A natural generalization of this notion is to say that $M$ is
``shifted" $m$-generic at $q$ if there exists a module $M'$ which is $m$-generic at $q$ and such that $M'|_{G(q)}\cong M^{[e]}|_{G(q)}$ for some
$e\geq 0$. Thus, $\opH^m(G(q),M)\cong \opH^m(G(q),M')\cong\opH^m_\gen(G,M')$. 
Our paper shows that many modules may be fruitfully
regarded as shifted $m$-generic at $q$, when it is unreasonable or false that they are $m$-generic
at $q$. The digit bounding results discussed above, which mesh especially well with the generic
cohomology theory, provide the main tool for finding such modules in non-trivial cases,  and this is the strategy for the proof of Theorem \ref{shiftedGeneric}.
In fact, Theorem \ref{shiftedGeneric} shows that often one can obtain the additional isomorphism
$\opH^m_\gen(G,M')\cong\opH^m(G,M')$, an attractive property for computations. 

We thank Chris Bendel, Dan Nakano, and Cornelius Pillen for remarks on an early draft of this paper, and
for supplying several references to the literature.

\section{Some preliminaries}
 Fix an irreducible root system
$\Phi$ with positive (resp., simple) roots $\Phi^+$ (resp., $\Pi$) selected.\footnote{The assumption that
$\Phi$ is irreducible is largely a convenience. The reader can easily extend to the general case, i.~e.,
when the group $G$ below is only assumed to be semisimple. } Let $\alpha_0\in\Phi^+$ be
the maximal short root, and let $h=(\rho,\alpha_0^\vee)+1$ be the Coxeter number of $\Phi$ (where $\rho$ is
the half sum of the positive roots). Write $X$ for the full weight lattice of $\Phi$, and let $X^+\subset X$ be the set of dominant weights determined by $\Pi$.

Now fix a prime $p$. For
a positive integer $b$, let $X^+_b:=\{\lambda\in X^+\,|\, (\lambda,\alpha^\vee)<p^b, \forall \alpha\in\Pi\,\}$
be the set of $p^b$-restricted dominant weights.  At times it is useful to regard the $0$ weight as (the only)
$p^0$-restricted dominant weight.

Let $G$ be a simple, simply connected algebraic group, defined and split over
a prime field ${\mathbb F}_p$ and having root system $\Phi$.  Fix a maximal split
torus $T$, and let $B\supset T$ be the Borel subgroup determined the negative roots $-\Phi^+$.
For $\lambda\in X^+$, $L(\lambda)$ denotes the irreducible rational $G$-module of highest
weight $\lambda$. 
If $F:G\to G$ is the Frobenius morphism, then, for any positive integer $b$, let $G_b=\Ker(F^b)$
be the (scheme theoretic) kernel of $F^b$. Thus, $G_b$ is a normal, closed (infinitesimal) subgroup of $G$.
Similar notations are used for other closed subgroups of $G$. 

The representation and cohomology theory for linear algebraic groups (especially semisimple groups and
their important subgroups) is extensively developed in Jantzen's book \cite{Jantzen}, with which we assume
the reader is familar. We generally
follow his notation (with some small modifications).\footnote{The reader should keep in mind that
$\Ext^m_G(L(\lambda),L(\mu))\cong \Ext^m_G(L(\mu),L(\lambda))$ and a similar statement holds
for $G(q)$. Often we write the $L(\mu)$ on
the left, because $\mu$ sometimes plays a special role (with restrictions of some kind), and taking $\mu=0$
gives $\opH^m(G,L(\lambda))$. But we are not always consistent, as in some places where it is more
convenient to have $L(\mu)$ on the right.}
If $M$ is a rational $G$-module and $b$ is a non-negative integer, write $M^{[b]}$ for the rational $G$-module
obtained by making $g\in G$ act through $F^b(g)$ on $M$. If $M$ already has the form $M\cong N^{[r]}$
for some $r\geq b$, write $M^{[-b]}:=N^{[r-b]}$.

Let $\ind_B^G$ be the induction functor from the category of rational $B$-modules to rational $G$-modules.
(See \S4 for a brief discussion of induction in general.) Given $\lambda\in X$, we denote the corresponding one-dimensional rational $B$-module
also by $\lambda$, and write $\opH^0(\lambda)$ for $\ind_B^G\lambda$. Then $\opH^0(\lambda)\not=0$
if and only if $\lambda\in X^+$; when $\lambda\in X^+$, $\opH^0(\lambda)$ has irreducible socle
$L(\lambda)$ of highest weight $\lambda$, and formal character $\ch\opH^0(\lambda)$ given by Weyl's character formula at the dominant weight $\lambda$. In most circumstances, especially when
regarding $\opH^0(\lambda)$ as a co-standard (i.e., a dual Weyl) module in the highest weight category of rational
$G$-modules, we denote $\opH^0(\lambda)$ by $\nabla(\lambda)$.\footnote{ Similarly, $\nabla_\zeta(\lambda)$ will
denote the standard module of highest weight $\lambda$ in a category of modules for a quantum
enveloping algebra $U_\zeta$; see Lemma \ref{filtrationlemma}.} Given $\lambda\in X$, let $\lambda^*:=-w_0(\lambda)$,
where $w_0$ is the longest element in the Weyl group $W$ of $\Phi$. If $\lambda\in X^+$, then $\lambda^*\in X^+$ 
is just the image of $\lambda$ under the opposition involution. For $\lambda\in X^+$, put $\Delta(\lambda)=
\nabla(\lambda^*)^*$, the dual of $\nabla(\lambda^*)$. In other words, $\Delta(\lambda)$ is the Weyl
module for $G$ of highest weight $\lambda$. Of course, $L(\lambda)^*=L(\lambda^*)$. 

 For $i\geq 0$,
let $R^i\ind_B^G$ be the $i$th derived functor of $\ind_B^G$. Then $R^i\ind_B^G=0$ for $i>|\Phi^+|$.

We will need another notion of the magnitude of a weight. If $b$ is a nonnegative integer, $\lambda\in X$ is called \emph{$b$-small} if $|(\lambda,\alpha^\vee)|\leq b$ for all $\alpha\in\Phi^+$. If $\lambda\in X^+$, $\lambda$ is $b$-small if and only if $(\lambda,\alpha^\vee_0)\leq b$. We say a (rational) $G$-module is \emph{$b$-small} provided all of its all of its weights are $b$-small. Equivalently,
it is $b$-small provided its maximal weights (in the dominance order) are $b$-small.
In particular, if $\lambda\in X^+$ is $b$-small, then any highest weight module $M$ with highest weight $\lambda$, e.~g., $L(\lambda)$, $\nabla(\lambda)$, or $\Delta(\lambda)$,  is also $b$-small. We make some elementary remarks about small-ness.

\begin{lem}\label{LenLem}Let $\nu$ be any dominant weight and let $b,b',r,u$ be non-negative intergers.
\begin{enumerate}\item[(a)] If $b>0$, assume $u\geq [\log_p b]+1$, where $[\,\,]$ denotes the greatest
integer function. Then $b\leq p^u-1$, and, if $\nu$ is $b$-small, $\nu$ is $p^u$-restricted.
\item[(b)] If $\nu$ is $p^r$-restricted, then $\nu$ is $(h-1)(p^r-1)$-small.
\item[(c)] Let $M$ and $N$ be two highest weight modules for $G$ with highest weights $\nu$, $\mu$. If
$\mu$ and $\nu$ are $b$-small and $b'$-small, respectively,   the tensor product $M\otimes N$ is $(b+b')$-small.
\item[(d)] If $\lambda,\mu\in X^+_b$ are both $p^b$-restricted, then all composition factors of $L(\lambda)\otimes L(\mu)$ are $p^{b+[\log_p(h-1)]+2}$-restricted. If, in fact, $p\geq 2h-2$ all the composition factors of $L(\lambda)\otimes L(\mu)$ are $p^{b+1}$-restricted.
\end{enumerate}
\end{lem}
\begin{proof} First we prove (a). The case $b=0$ is clear, so assume $b>0$. If $b\geq p^u$, then
$\log_pb\geq [\log_pb]+1$, which is false. Thus, $b\leq p^u-1$. If $\nu$ is $b $-small, then
$(\nu,\alpha^\vee)\leq b$ for each $\alpha\in\Phi^+$, so $\nu$ is $p^u$-restricted. Hence,  (a) holds.

For (b) we note $(\nu,\alpha_0^\vee)\leq ((p^n-1)\rho,\alpha_0^\vee)\leq (p^r-1)(h-1)$. This proves (b) since $\nu\in X^+$ is dominant.

For (c), the highest weight of $M\otimes N$ is $(b+b')$-small. Since any other weight of $M\otimes N$ is obtained by subtracting positive roots, the statement follows.

To prove (d), note that (b) implies that $\lambda$ and $\mu$ are $(h-1)(p^b-1)$-small. Thus by (c) all composition factors of $L(\lambda)\otimes L(\mu)$ are $2(h-1)(p^b-1)$-small. Then by (a), all composition factors of $L(\lambda)\otimes L(\mu)$ are $p^e$-restricted, where \begin{align*}e&=[\log_p(2(h-1)(p^b-1))]+1\\
&\leq [\log_p2+\log_p(p^b-1)]+[\log_p(h-1)]+2\\
&\leq b+[\log_p(h-1)]+2.\end{align*} The case $p\geq 2h-2$ follows similarly. \end{proof}

\section{Bounding weights} Let $U=R_u(B)$ be the unipotent radical of $B$, and let $u$ be the Lie
algebra of $U$.

\begin{lem}\label{Lem0}For any non-negative integer $m$, the $T$-weights in the ordinary cohomology
space $\opH^m(u,k)$ are $3m$-small (and they are sums of positive roots).
\end{lem}
\begin{proof} The $T$-weights in $\opH^m(u,k)$ are included among the  $T$-weights of the exterior power $\bigwedge^m(u^*)$ appearing in the the Koszul complex computing $\opH^\bullet(u,k)$. Hence, they are sums of $m$ positive
roots. Since each positive root is $3$-small, these weights are $3m$-small.\end{proof}

Recall that, for $r\geq 1$, $U_r$ is the Frobenius kernel of $F^r|_U$. 
\begin{lem}\label{Lem1} For any non-negative integer $m$, the $T$-weights of $\opH^m(U_1,k)$ are $3mp$-small.\end{lem}
\begin{proof}By \cite[I.9.20]{Jantzen} and \cite[(1.2)(b)]{FP86} there are spectral sequences
$$\begin{cases}
p=2: E_2^{i,j}:=S^i(u^*)^{[1]}\otimes \opH^i(u,k) \implies \opH^{i+j}(U_1,k);\\
p\not=2: E_2^{2i,j}:=S^i(u^*)^{[1]}\otimes \opH^j(u,k)\implies \opH^{2i+j}(U_1,k).\end{cases}$$
Suppose that $p=2$. A weight $\nu$ in $ \opH^m(U_1,k)$ is a weight in $E^{i,j}_2$ for some
$i+j=m$. Using Lemma \ref{Lem0}, the largest value of $(\nu,\alpha_0)$ clearly occurs when $i=m$ and $j=0$. Since of weights of $S^m(u)^{[1]}$ are given as a sum $\sum_{k=1}^mp\alpha_k$, the weight $\nu$ is $3mp$-small.
Similarly, when $p>2$, the weight $\mu$ is $3(m/2)p$-small, so certainly $3mp$-small also.
\end{proof}
 
\begin{lem}\label{Lem2} For any $r>0$ and non-negative integer $m$, the $T$-weights of $\opH^m(U_r,k)$ are $3mp^r$-small.\end{lem}

\begin{proof} We use the Lyndon--Hochschild--Serre spectral sequence
$$E_2^{i,j}:=\opH^i(U_r/U_1,\opH^j(U_1,k))\implies \opH^{i+j}(U_r,k).$$
The $E_2^{i,j}$-term has the same weights for $T$ as the $T$-module \[\opH^i(U_r/U_1,k)\otimes \opH^j(U_1,k)\cong
\opH^i(U_{r-1},k)^{[1]}\otimes \opH^j(U_1,k).\] The weights on the left hand tensor factor are $p(3ip^{r-1})=
3ip^r$-small. On the right hand side, the weights are $3jp$-small. Adding these together, the worst case
occurs for $i=m$, and the lemma follows. \end{proof}

The following is immediate:
\begin{cor}\label{Cor1} Suppose the weight $\lambda$ is $b$-small. Then the weights of $\opH^m(B_r,\lambda)\cong
(\opH^m(U_r,k)\otimes\lambda)^{T_r}$ are $(3mp^r+b)$-small. Moreover, the weights of $\opH^m(B_r,\lambda)^{[-r]}$
are $(3m +[b/p^r])$-small, where $[\,\,\,]$ denotes the greatest integer function.\end{cor}

\begin{thm}\label{MsmallToHsmall} Let $m$ be a non-negative integer, and let $r,b$ be positive integers.
Let $M$ be a $b$-small $G$-module. Then the $G$-module $\opH^m(G_r,M)^{[-r]}$ is  $(3m+[b/p^r])$-small.

In particular, if $M$ is $p^r$-restricted, then $\opH^m(G_r,M)^{[-r]}$ is $(3m +[(h-1)(p^r-1)/p^r])\leq (3m+h-2)$-small.\end{thm}

\begin{proof} We will show the statement holds when $M$ is an induced module with highest $b$-small weight $\lambda$; thence we deduce that the statement holds for $L(\lambda)$, so the statement follows for all $M$, since it holds for its composition factors.

By \cite[II.12.2]{Jantzen}, there is a first quadrant spectral sequence
$$E_2^{i,j}:=R^i\ind_B^G\opH^j(B_r,\lambda)^{[-r]}\implies \opH^{i+j}(G_r,\opH^0(\lambda))^{[-r]}.$$
Any weight in $\opH^m(G_r,\opH^0(\lambda))^{[-r]}$ is a weight of $E_2^{i,j}$ for
some $i,j$ with $i+j=m$, hence a weight of $\opH^i(\mu)$ for some $\mu\in \opH^j(B_r,\lambda)^{[-r]}$. So it suffices to show that any weight of $\opH^i(\mu)$ for $\mu\in \opH^j(B_r,\lambda)^{[-r]}$ is $(3m+[b/p^r])$-small.

By Corollary \ref{Cor1},  a weight $\mu$ of $\opH^j(B_r,\lambda)^{[-r]}$ is $(3j+[b/p^r])$-small; hence it is also $b':=(3m+[b/p^r])$-small.
 Choose $w\in W$ so that $w\cdot\mu\in X^+-\rho$. If $w\cdot\mu$ is not in $X^+$ then $R^i\Ind_B^G\mu=0$. Hence we may assume that $w\cdot\mu\in X^+$.
Now, $w\cdot\mu=w(\mu)+w\rho-\rho\leq w(\mu)$.
 Since $w\cdot\mu\in X^+$, $w\cdot\mu$ is $b'$-small if and only if $(w\cdot\mu,\alpha_0^\vee)\leq b'$. But
 $$(w\cdot\mu,\alpha_0^\vee)\leq(w(\mu),\alpha_0^\vee)=(\mu,w^{-1}\alpha_0^\vee)\leq b'.$$

Now if $L(\nu)$ is a composition factor of $R^i\ind_B^G \mu$, the strong linkage principle \cite[II.6.13]{Jantzen} implies
 $\nu\uparrow w\cdot\mu$ and is in particular $b'$-small.

Thus we have proved the statement in the case $M=\opH^0(\lambda)$.

For the general case, we apply induction on $m$. We have a short exact sequence $0\to L(\lambda)\to \opH^0(\lambda)\to N\to 0$ where the $G$-module $N$ has composition factors whose high weights are less than $\lambda$ in the dominance order and are therefore $b$-small. Associated to this sequence is a long exact sequence of which part is
 \[\opH^{m-1}(G_r,N)^{[-r]}\to \opH^m(G_r,L(\lambda))^{[-r]}\to \opH^m(G_r,\opH^0(\lambda))^{[-r]}\]
 so that any $G$-composition factor of the middle term must be a $G$-composition factor of one of the outer terms. Now, the composition factors of the rightmost term are $(3m +[b/p^r])$-small by the discussion above. Since $N$ has composition factors with high weights less than $\lambda$ in the dominance order, these weights are $(3(m-1) +[b/p^r])$-small by induction, and are in particular $(3m+[b/p^r])$-small. Thus the weights of the middle term are also $(3m +[b/p^r])$-small.

 This proves the statement in the case $M=L(\lambda)$. The  case for all $b$-small modules $M$ now follows since it is true for each of its composition factors.

 For the last statement, we use Lemma \ref{LenLem}(b).
\end{proof}
 
 The following corollary for general $m$ follows from the previous theorem. For $m=1$, it is proved in
 \cite[Prop. 5.2]{BNP04-Frob}. 
 
\begin{cor}\label{extForGr}Let $\lambda,\mu\in X^+_r$. For any $r\geq 1$, the weights of $\Ext^m_{G_r}(L(\lambda),L(\mu))^{[-r]}$ are $(3m+2h-3)$-small. If $m=1$, the weights  of $\Ext^1_{G_r}(L(\lambda),L(\mu))^{[-r]}$ are $(h-1)$-small. \end{cor}

\begin{rems} (a) The result \cite[Prop 5.2]{BNP04-Frob} quoted above also gives, for $m>1$, an integer $b$ such that the
weights in $\Ext^m_{G_r}(L(\lambda),L(\mu))^{[-r]}$ are $b$-small. However, $b$ is multiplicative in
$h$ and $m$, and so is weaker than Corollary \ref{extForGr} 
for large $m$ and $h$ (for instance if $m,h\geq 4$). For $p>h$ and $m\geq 2$ note that our bound coincides with that given in \cite[Prop. 5.3]{BNP04-Frob} using the improvements from (c) below.

 (b)  It is interesting to ask when $\Ext^m_{G_r}(L(\lambda),L(\mu))^{[-r]}$ for $\lambda,\mu\in X^+_r$ has a good filtration. Even for $r=1$, there are examples due to Peter Sin (for instance, see \cite[Lem. 4.6]{Sin94}) showing that for small $p$ this question can have a negative answer. Obviously, if $p\geq 3m+3h-4$ (or 
 $p\geq 2h-2$ in case $m=1$), then $\Ext^m_{G_r}(L(\lambda),L(\mu))^{[-r]}$ has highest weights in 
 the lowest $p$-alcove, so it trivially has a good filtration.\footnote{The first two authors expect to prove in a later paper that a good filtration of $\Ext^m_{G_r}(L(\lambda),L(\mu))^{[-r]}$ always exists for restricted regular weights when $r=1$, $p\geq 2h-2$ and the Lusztig character formula holds for all irreducible modules with restricted highest weights.}

(c) The reader may check that many of the results in this section can be improved under certain mild conditions. For instance, if $\Phi$ is not of type $G_2$, its roots are all $2$-small. In this instance, wherever we have `$3m$' it can be replaced with `$2m$'. In addition, if $p>2$ the last sentence of the proof of Lemma \ref{Lem2} shows that one can replace $m$ with $[m/2]$. The same statement follows for most formulas in the remainder of the paper, however, we will not elaborate further in individual cases.  \end{rems}

\section{Relating $G(q)$-cohomology to $G$-cohomology} Inspired by work \cite{BNP01}, \cite{BNP04}, \cite{BNP02}, \cite{BNP06}, and \cite{BNP11}, Theorem
\ref{existsAnF}  establishes
an important procedure for describing $G(q)$-cohomology in terms of $G$-cohomology. This result will
be used in \S5 to prove the digit bounding results mentioned in the Introduction.

Before stating the theorem, we review some elementary results. The coordinate algebra $k[G]$ of $G$ is a left $G$-module ($f\mapsto g\cdot f, x\mapsto f(xg)$, $x,g\in G$, $f\in k[G]$) and a right $G$-module ($f\mapsto f\cdot g, x\mapsto f(gx)$, $x,g\in G$,$g\in k[G]$). Given a closed subgroup $H$ of $G$ and a rational $H$-module,
the induced module $\ind_H^G(V):=\Map_H(G,V)$ consists of all morphisms $f:G\to V$ (i.~e., morphisms of
the algebraic variety $G$ into the underlying variety of a finite dimensional subspace of $V$), which are
$H$-equivariant in the sense that $f(h.g)=h.f(g)$ for all $g,h\in G$. If $x\in G, f\in\ind_H^GV$, $x\cdot f\in
\ind_H^GV$, making $\ind_H^GV$ into a rational $G$-module (characterized by  the  property that $\ind_H^G$ is the right adjoint of the restriction functor $\res^G_H:G$--mod $\to$ $H$--mod). If $G/H$ is an affine variety (e.~g., if $H$ is a finite subgroup), $\ind_H^G$ is an exact functor \cite[Thm. 4.3]{CPS77}, which formally takes injective $H$-modules to injective $G$-modules. Thus, 
$\opH^\bullet(H,V)\cong\opH^\bullet(G,\ind_H^GV)$ for any rational $H$-module.  

Let $q=p^r$ for some prime integer $p$ and positive integer $r$, and let $m$ be a fixed non-negative
integer which will serve as the cohomological degree. As in \S2, let $G$ be the simple, simply connected algebraic group
defined and split over ${\mathbb F}_p$ with root system $\Phi$.

 \begin{lem}\label{Kop} (\cite{Kop84})
 Consider the coordinate algebra $k[G]$ of $G$ as a rational $(G\times G)$-module with action
  $$((g,h)\cdot f)(x)=f(h^{-1}xg),\,\,g,h,x\in G,
 f\in k[G].$$
 Then $k[G]$ has an increasing $G\times G$-stable  filtration $0\subset \mF'_0\subset\mF'_1\subset \cdots$ in which, for $i\geq 1$, $\mF'_i/\mF'_{i-1}\cong
  \nabla(\gamma_i)\otimes \nabla(\gamma^*_i)$, $\gamma_i\in X^+$, and $\cup_i\mF'_i=k[G]$. Each dominant weight $\gamma\in X^+$
  appears exactly once in the sequence $\gamma_0,\gamma_1,\gamma_2, \cdots$. \end{lem}

\begin{thm}\label{filtration} (\cite[Prop. 2.2 \& proof]{BNP11}) (a) The induced module $\ind_{G(q)}^G(k)$ is isomorphic to the pull-back of the $G\times G$-module
$k[G]$ above through the map $G\to G\times G$, $g\mapsto (g,F^r(g))$.

(b) In this way, $\ind_{G(q)}^G k$ inherits an increasing $G$-stable filtration $0\subset \mF_0\subset \mF_1\subset \cdots$ with $\bigcup \mF_i
=\ind_{G(q)}^Gk$,  in which, for $i\geq 1$,
$\mF_i/\mF_{i-1}\cong\nabla(\gamma_i)\otimes\nabla(\gamma_i^*)^{[r]}$. Moreover,
each dominant weight $\gamma\in X^+$ appears exactly once in the sequence  $\gamma_0,\gamma_1,\cdots$.
\end{thm}

Following \cite{BNP11}, put $\sG_r(k):=\ind_{G(q)}^Gk$, with $q=p^r$. The filtration $\mF_\bullet$ of
the rational $G$-module $\sG_r(k)$ arises from the increasing $G\times G$-module filtration $\mF'_\bullet$ of $k[G]$ with
sections $\nabla(\gamma)\otimes\nabla(\gamma^*)$. Since these latter modules are all co-standard modules for $G\times G$, their order in $\mF'_\bullet$
can be manipulated, using the fact that
\begin{equation}\label{vanishing}
\Ext^1_{G\times G}(\nabla(\gamma)\otimes\nabla(\gamma^*),\nabla(\mu)\otimes\nabla(\mu^*))=0,\quad
{\text{\rm unless $\mu<\gamma$ (and $\mu^*<\gamma^*$).}}\end{equation}
  Thus, for any non-negative integer $b$, there is a (finite dimensional)
$G$-submodule $\sG_{r,b}(k)$ of $\sG_r(k)$ which has an increasing $G$-stable filtration with sections precisely the $\nabla(\gamma)\otimes
\nabla(\gamma^*)^{[r]}$ and $(\gamma,\alpha_0^\vee)\leq b$, and each such $\gamma$ appearing with
multiplicity 1. Such a submodule may be constructed from a corresponding (unique) $G\times G$--submodule of
$k[G]$ with corresponding sections $\nabla(\gamma)\otimes\nabla(\gamma^*)$. With this construction,
the quotient $\sG_r(k)/\sG_{r,b}(k)$ has a $G$-stable filtration with sections $\nabla(\gamma)\otimes\nabla(\gamma^*)^{[r]}$,
$(\gamma,\alpha_0^\vee)>b$. More precisely, using (\ref{vanishing}) again and setting $\sG_{r,-1}(k)=0$,
we have the following result.

\begin{lem}\label{filtrationlemma} For $b\geq 0$,
$$\sG_{r,b}(k)/\sG_{r,b-1}(k)\cong\bigoplus_{\lambda\in X^+, (\lambda,\alpha^\vee_0)=b}\nabla(\lambda)\otimes
\nabla(\lambda^*)^{[r]}.$$
Also, $\bigcup_{b\geq 0}\sG_{r,b}(k) =\sG_r(k).$
\end{lem}

 We will usually abbreviate $\sG_r(k)$ to $\sG_r$ and $\sG_{r,b}(k)$ to $\sG_{r,b}$ for $b\geq -1$. We
 remark that $\sG_r(k)$ is in some sense already  an abbreviation, since it depends on the characteristic $p$ of $k$.

For $\lambda,\mu\in X_r^+$, set $\Ext^m_{G(q)}(L(\lambda),L(\mu))$. Because the induction functor
$\ind_{G(q)}^G$ is exact from the category of $kG(q)$-modules to the category of rational $G$-modules,
\begin{equation}\label{induction} \Ext^m_{G(q)}(L(\lambda)\otimes L(\mu^*),k)\cong \Ext^m_G(L(\lambda)\otimes L(\mu^*),\sG_r),\end{equation}
where $q$ is $p^r$.

The following result provides an extension beyond the $m=1$ case treated in \cite[Thm. 2.2]{BNP02}.   A similar result in the cohomology case is given with a bound on $p$ (namely, $p\geq  (2m+1)(h-1)$) for any $m$ by \cite[proof of Cor. 7.4]{BNP01}. Our result, for $\Ext^m$, does not require
any condition on $p$. 

\begin{thm}\label{existsAnF} Let $b\geq 6m+6h-8$, independently of $p$ and $r$, or, more generally,
 \[b\geq b(\Phi,m,p^r):=\left[\frac{3m+3h-4}{1-1/p^r}\right],\]
 when $p$ and $r$ are given. Then, for any $\lambda,\mu\in X^+_r$, we have
 \begin{equation}\label{decomposition}
\Ext^m_{G(q)}(L(\lambda),L(\mu))\cong\Ext^m_G(L(\lambda),L(\mu)\otimes \sG_{r,b}).\end{equation}
(Recall $q=p^r$.) In addition,
\begin{equation}\label{vanishing}
\Ext^n_G(L(\lambda),L(\mu)\otimes\nabla(\nu)\otimes\nabla(\nu^*)^{[r]})=0,\quad\forall n\leq m,  \forall\nu\in X^+
\,\,{\text{\rm satisfying}}\, (\nu,\alpha_0^\vee)>b.\end{equation}
 \end{thm}

\begin{proof}It suffices to show that
\begin{equation}\label{vanishingresult} b\geq b(\Phi,m,p^r)\implies \Ext^n_G(L(\lambda),L(\mu)\otimes
\sG_r/\sG_{r,b})=0,\,\forall n\leq m.\end{equation}
Suppose that (\ref{vanishingresult}) fails. Then for some $\nu$ with $(\nu,\alpha^\vee_0)>b$ and some
non-negative integer $n\leq m$, we must have
$\Ext^n_G(L(\lambda),L(\mu)\otimes\nabla(\nu)\otimes\nabla(\nu^*)^{[r]})\not=0$. (That is, (\ref{vanishing})
fails.)
For some composition factor $L(\xi)\cong L(\xi_0)\otimes L(\xi')^{[r]}$ ($\xi_0\in X^+_r$, $\xi'\in
X^+$) of $\nabla(\nu)$, we obtain
\begin{equation}\label{first}\Ext^n_G(\Delta(\nu)^{[r]}\otimes L(\lambda), L(\mu)\otimes L(\xi_0)\otimes L(\xi')^{[r]})\not=0\end{equation}
by rearranging terms. Here we use the fact that $\nabla(\nu^*)^*\cong\Delta(\nu)$, the Weyl module
of highest weight $\nu$. To compute the left-hand side of (\ref{first}) we use a Lyndon-Hochschild-Serre spectral sequence
involving the normal subgroup $G_r$. The $E_2$-page is given by
\begin{equation}\label{spectralsequence} E_2^{i,j}=\Ext_G^i(\Delta(\nu),\Ext_{G_r}^j(L(\lambda)\otimes L(\mu^*), L(\xi_0))^{[-r]}\otimes
L(\xi')),\quad i+j=n.\end{equation}

 Using Lemma \ref{LenLem}(b)(c), we see that the module $L(\lambda^*)\otimes L(\mu)\otimes L(\xi_0)$ is $3(h-1)(p^r-1)$-small; thus by Theorem \ref{MsmallToHsmall} we have $\Ext^j_{G_r}(L(\lambda)\otimes L(\mu^*),L(\xi_0))^{[-r]}=\Ext^j_{G_r}(k,L(\lambda^*)\otimes
L(\mu)\otimes L(\xi_0))^{[-r]}$ is $(3j+3h-4)$-small, and, in particular, is $(3m+3h-4)$-small.

Put $x=(\nu,\alpha_0^\vee)>b$. Then, as $L(\xi)$ is a composition factor of $\nabla(\nu)$, $\xi=\xi_0+p^r\xi'$ is $x$-small. Clearly $p^r\xi'$ is $x$-small also, thus $\xi'$ is $[x/p^r]$-small. So, the composition factors of  $\Ext^j_{G_r}(L(\lambda)\otimes L(\mu^*),L(\xi_0))^{[-r]}\otimes L(\xi')$ are $([x/p^r]+3m+3h-4)$-small. Recall
the fact, for general $\nu,\omega\in X^+$, that $\Ext^\bullet_G(\Delta(\nu),L(\omega))\not=0$ implies
that $\nu\leq \omega$. For $\nu$ as above and $L(\omega)$ a
composition factor of $\Ext^j_{G_r}(L(\lambda)\otimes L(\mu^*),L(\xi_0))^{[-r]}\otimes L(\xi')$, it follows that
$\nu\leq \omega$. Hence, $x=(\nu,\alpha_0^\vee)\leq (\omega,\alpha_0^\vee)\leq 3m +3h-4+[x/p^r]$.
Rearranging this gives
$$ x\leq\left[\frac{3m+3h-4}{1-1/p^r}\right]\leq b,$$
a contradiction. This proves (\ref{decomposition}) and (\ref{vanishing}).

For the remaining part of the theorem, just note that the smallest value of $p^r$ is $2$. Hence, the
largest value of  $v(\Phi,m,p^r)$ is $6m+6h-8$. 
\end{proof}



\section{Digits and cohomology}

Any $\lambda\in X^+$ has a $p$-adic expansion $\lambda=\lambda_0+p\lambda_1+\dots+p^r\lambda_r+\dots$ where each $\lambda_i$ is $p$-restricted. We refer to each pair $(i,\lambda_i)$ as a digit of $\lambda$. We say the $i$th digit of $\lambda$ is $0$ if $\lambda_i=0$. Clearly $\lambda$ has finitely many nonzero digits. Let also $\mu\in X^+$. We say $\lambda$ and $\mu$ agree on a digit if there is a natural number $i$ with $\lambda_i=\mu_i$. We say $\lambda$ and $\mu$ differ on $n$-digits if $|\{i:\lambda_i\neq\mu_i\}|=n$.

The theorem below requires the following result.
\begin{lem}[{\cite[Prop. 3.1]{BNP06}}] \label{tensors}Let $\lambda,\mu\in X_r^+$ and $M$ a finite dimensional rational $G$-module whose (dominant) weights are $(p^r-1)$-small.  Then \[\Hom_{G_r}(L(\lambda),L(\mu)\otimes M)= \Hom_G(L(\lambda),L(\mu)\otimes M).\]
Hence, the left-hand side has trivial $G$-structure.\end{lem}

We can now prove the following ``digit bounding" theorem. It both answers the open question
\cite[Question 3.10]{SteSL3} in a strong way, and paves the way for for this rest of this section.

\begin{thm}\label{digits} Given an irreducible root system $\Phi$ and a non-negative integer $m$, there is a constant $\delta=\delta(\Phi,m)$, so that if $\lambda,\mu, \nu\in X^+$, and $\nu$ is $(3m+2h-2)$-small, then \begin{enumerate}\item\[\Ext_G^n(L(\lambda),L(\nu)\otimes L(\mu))\neq 0\] for some $n\leq m$ implies $\lambda$ and $\mu$ differ in at most $\delta$ digits; and
\item \[\Ext_G^m(L(\lambda),L(\mu))\neq 0\] implies $\lambda$ and $\mu$ differ in at most $\delta-\phi$ digits,\end{enumerate}
where $\phi=[\log_p(3m+2h-2)]+1$.
\end{thm}
\begin{proof}We prove both statements together by induction on $m$. Set $b:= 3m+2h -2$ and $u:=\phi$.
Thus, $u=[\log_p b]+1$.  Let $\lambda^\circ
:=\lambda_0+\cdots + p^{u-1}\lambda_{u-1}$, so that $\lambda=\lambda^\circ + p^u\lambda'$, for $\lambda'\in X^+$; write $\mu$ similarly. By Lemma \ref{LenLem}, the $b$-small weight $\nu$ is $p^u$-restricted. In fact, $b\leq p^u-1$.

The case $m=0$ follows easily from Lemma \ref{tensors}, with $\delta(\Phi,0):=\phi$:

For statement (i), we have
$$\opH:=\Hom_G(L(\lambda),L(\mu)\otimes L(\nu))=\Hom_G(L(\lambda'),\Hom_{G_u}(L(\lambda^\circ),L(\mu^\circ)\otimes L(\nu))^{[-u]}\otimes L(\mu'))$$
with $\opH$ assumed to be nonzero. As $\nu$ is $(p^u-1)$-small,  Lemma \ref{tensors} implies that the module $\Hom_{G_u}(L(\lambda^\circ),L(\mu^\circ)\otimes L(\nu))\cong \Hom_{G}(L(\lambda^\circ),L(\mu^\circ)\otimes L(\nu))$ and so has a trivial $G$-structure. Thus,
 \[\opH=\Hom_G(L(\lambda'),L(\mu'))\otimes \Hom_{G}(L(\lambda^\circ),L(\mu^\circ)\otimes L(\nu)).\]
 If this expression is nonzero, then $\lambda'=\mu'$ and $\lambda^\circ$, $\mu^\circ$ can differ in at most all of their $u=\phi$ places, so that $\lambda$ and $\mu$ can differ in at most $\phi$ places. Statement (ii) is trivial.
This completes the $m=0$ case.

Assume we have found $\delta(\Phi,i)$ for all $i<m$ such that the theorem holds when $i$ plays the role of $m$ and $\delta(\Phi,i)$ plays the role of $\delta(\Phi,m)$.
We claim that the theorem holds at $m$ if we set $\delta=\delta(\Phi,m):=2\phi+\max_{i<m}\delta(\Phi,i)$.

Suppose otherwise. Let $b=(3m+2h-2)$. Then either (i) fails with $n=m$, namely,
\begin{equation}\label{failureofi}\begin{cases}
\Ext^m_G(L(\lambda),L(\nu)\otimes L(\mu))\not=0,\,\,{\text{\rm for some $\lambda,\mu,\nu\in X^+$,}}\\
{\text{\rm with $\nu$ $b$-small and $\lambda$ and $\mu$ differing in more than $\delta$-digits,}}\end{cases}
\end{equation}
or (ii) fails, namely,
\begin{equation}\label{failureofii}
\Ext^m_G(L(\lambda),L(\mu))\not=0,\,\,{\text{\rm $\lambda$ and $\mu$ differing in more than $\delta-\phi$ digits.}}
\end{equation}
Let $\lambda,\mu,\nu\in X^+_s$ be a such a counterexample with $s$ minimal (where in (\ref{failureofii}),
we take $\nu=0$ and, in (\ref{failureofi}), $\nu$ is $b$-small).  We continue to write $\lambda=\lambda^\circ+p^u \lambda'$,
where $\lambda^\circ\in X^+_{u}$, $\lambda'\in X^+$. Similarly for $\mu$.

We investigate $\lambda$ and $\mu$ using the Lyndon--Hochschild--Serre spectral sequence for the normal (infinitesimal)
group $G_{u}\triangleleft G$.  First, suppose
that
\[
 E_2^{m-i,i}:=\Ext^{m-i}_{G}(L(\lambda'),\Ext_{G_u}^{i}(L(\lambda^\circ),L(\nu)\otimes L(\mu^\circ))^{[-u]}\otimes L(\mu'))\not=0,\]
 for
some positive integer $0<i\leq m$.  Then $\Ext^{m-i}_G(L(\lambda'), L(\tau)\otimes L(\mu'))\not=0$ for
some composition factor $L(\tau)$ of $\Ext_{G_u}^{i}(L(\lambda^\circ),L(\nu)\otimes L(\mu^\circ))^{[-u]}$. By Lemma \ref{LenLem}(b),(c), all composition factors of $L(\nu)\otimes L(\mu^\circ)\otimes L(\lambda^\circ)^*$ are $(p^u-1)(2h-2)+b\leq (p^u-1)(2h-1)$-small. (Note that $\nu$ is $(p^u-1)$-small.) Thus, by Theorem \ref{MsmallToHsmall}, $\tau$ is $(3i+2h-2)\leq (3(m-1)+2h-2)$-small.

By induction, $\mu'$ differs from $\lambda'$ in at most $\delta(\Phi,m-i)$ digits. (Apply (i) with $n=m-i$ and $m-1$ playing the role of $m$.)
So the number of digits where $\lambda$ differs from $\mu$ is at most $\phi+\delta(\Phi,m-i)\leq \delta-\phi$
digits. This is a contradiction to (\ref{failureofi}) or to (\ref{failureofii}) in the $\nu=0$ case, as $\lambda$ was assumed
to differ from $\mu$ by more than $\delta-\phi$ digits.
Hence, we may assume that the terms $E_2^{m-i,i}=0$, for all positive integers $i\leq m$.

By assumption, $\Ext^m_G(L(\lambda),L(\nu)\otimes L(\mu))\neq 0$, so,
  $$E_2^{m,0}\cong \Ext_G^{m}(L(\lambda'),\Hom_{G_r}(L(\lambda^\circ),L(\nu)\otimes L(\mu^\circ))^{[-r]}\otimes L(\mu'))\neq 0.$$
  Now by Lemma \ref{tensors}, \[\Hom_{G_r}(L(\lambda^\circ),L(\nu)\otimes L(\mu^\circ))\cong \Hom_G(L(\lambda^\circ),L(\nu)\otimes L(\mu^\circ))\]
has trivial $G$-structure, and
$$E_2^{m,0}\cong \Ext_G^{m}(L(\lambda'), L(\mu')^{\oplus t})\cong (\Ext_G^{m}(L(\lambda'), L(\mu')))^{\oplus t}\neq 0,$$
where $t=\dim \Hom_G(L(\lambda^\circ),L(\nu)\otimes L(\mu^\circ))$ and $\lambda',\mu'\in X_{s-u}^+$.

By minimality of $s$ we have that $\lambda'$ and $\mu'$ differ in at most $\delta-\phi$ places. But $\mu^\circ$ and $\lambda^\circ$ differ in at most all their $u=\phi$ digits. So $\lambda$ and $\mu$ differ in at most $\delta$ places.

This is a contradiction when $\nu\neq 0$. So we may assume $\nu=0$ and
$$E_2^{m,0}\cong \Ext_G^{m}(L(\lambda'),\Hom_{G_u}(L(\lambda^\circ), L(\mu^\circ))^{[-u]}\otimes L(\mu'))\neq 0.$$
The non-vanishing forces $\lambda^\circ=\mu^\circ$. Since $u=\phi>0$, we have, by minimality of $s$,
that 
$$\Ext_G^{m}(L(\lambda'),L(\mu'))\neq 0 \implies \lambda' {\text{\rm and $\mu'$ differ in at most $\delta-u$ places}}.$$
 But as $\lambda$ and $\mu$ agree on their first first $u$ places, so $\lambda$ and $\mu$ differ in at most $\delta-u=\delta-\phi$ places.

This is a contradiction, and completes the proof of the theorem.
\end{proof}

\begin{rem}\label{special} The proof of the theorem implies that if $\Ext^1_G(L(\lambda),L(\mu))\neq 0$, then $\lambda$ and $\mu$ differ in at most $2+2[\log_p(h-1)]$ digits; indeed, following the proof carefully, one sees these digits can be found in a substring of length $2+2[\log_p(h-1)]$:

Let $r=[\log_p(h-1)]+1$. Write $\lambda=\lambda^\circ+p^r\lambda'=\lambda^\circ+p^r\lambda'^\circ+p^{2r}\lambda''$, with $\lambda^\circ,\lambda'^\circ\in X_r^+$, $\lambda',\lambda''\in X_r^+$ and take a similar expression for $\mu$.  If $\Ext^1_G(L(\lambda),L(\mu))\neq 0$ then either: $$\Ext^1_G(L(\lambda'),\Hom_{G_r}(L(\lambda^\circ),L(\mu^\circ))^{[-r]}\otimes L(\mu'))\neq 0,$$  which implies $\lambda^\circ=\mu^\circ$ (and we are done by induction, say on the maximum number of digits of $\lambda$ and $\mu$). Or,  the space $\Hom_G(L(\lambda'),\Ext^1_{G_r}(L(\lambda^\circ),L(\mu^\circ))^{[-r]}\otimes L(\mu'))\neq 0$. Now the weights of $\Ext^1_{G_r}(L(\lambda^\circ),L(\mu^\circ))^{[-r]}$ are $(h-1)$-small by Corollary \ref{extForGr}. By
Lemma \ref{LenLem}(a), $h-1\leq p^r-1$. By Lemma \ref{tensors}, the $G$-structure on $\Hom_{G_r}(L(\lambda'^\circ),L(\tau)\otimes L(\mu'^\circ))$ is trivial for a composition factor $L(\tau)$ of $\Ext^1_{G_r}(L(\lambda_0),L(\mu_0))^{[-r]}$. Hence, we must have $\Hom_G(L(\lambda''),L(\mu''))\neq 0$ and we can identify $\lambda''$ and $\mu''$. Thus, $\lambda$ and $\mu$ differ in their first $2r$ digits, as required. 
\end{rem}

\begin{thm}\label{digitsForG(q)}
For every $m\geq 0$ and irreducible root system $\Phi$, choose any
constant $m'$ so that $3m'+2h+2\geq 6m+6h-8$, and let $d=d(m)=d(\Phi,m)$ be an integer $\geq \delta(\Phi,m')$. Then for all prime powers $q=p^r$, and all $\lambda,\mu\in X_r^+$,
 $$\Ext^m_{G(q)}(L(\lambda),L(\mu))\not=0$$
  implies that $\lambda$ and $\mu$ differ in at most $d$ digits.
 \end{thm}
 \begin{proof} Let $b=6m+6h-8$.
 By Theorem \ref{existsAnF},
 \[E=\Ext^m_{G(q)}(L(\lambda),L(\mu))\cong \Ext^m_{G}(L(\lambda),L(\mu)\otimes \sG_{r,b}).\]
Also, $\sG_{r,b}$ has composition factors $L(\zeta)\otimes L(\zeta')^{[u]}$ with $\zeta, \zeta'$  in the set  $\Xi$ of $b$-small weights.

 Then we have $\Ext^m_{G}(L(\lambda)\otimes L(\zeta')^{[r]},L(\mu)\otimes L(\zeta))\neq 0$ for some $\zeta,\zeta'\in \Xi$. By our choice of $m'$, $m'\geq m$ and the weight $\zeta$ is $(3m'+2h-2)$-small.
The result now follows from Theorem \ref{digits}.
  \end{proof}

Next, define some notation in order to quote the main result of \cite{CPSK77}:
\begin{enumerate}
\item Let $t$ be the torsion exponent of the index of connection $[X:{\mathbb Z}\Phi]$.
\item For a weight $\lambda$, define $\bar\lambda=t\lambda$. Also let $t(\lambda)$ be the order of $\lambda$ in the abelian group $X/\Phi$. Let $t_p(\lambda)$ be the $p$-part of $t(\lambda)$. Of course one has $t_p(\lambda)\leq t(\lambda)\leq [X:{\mathbb Z}\Phi]$.
\item Let $c(\mu)$ for $\mu$ in the root lattice be the maximal coefficient in an expression of $\mu$ as a sum of simple roots; for $\mu=2\rho$ these values can be read off from \cite[Planches I--IX]{Bourb82}.
\item Let $c=c(\tilde\alpha)$ where $\tilde\alpha$ is the highest long root. One can also find the value of $c$ from [{\it loc.~cit.}].
\item For any $r\in \Z$ and any prime $p$, define $e(r)=[(r-1)/(p-1)]$; clearly $e(r)\leq r-1$.
\item For any $r\in \Z$ and any prime $p$, let $f(r)=[\log_p(|r|+1)]+2$; clearly $f(r)\leq [\log_2(|r|+1)]+2$.\end{enumerate}

Then the main result of \cite{CPSK77} is as follows:

\begin{thm}[{\cite[6.6]{CPSK77}}]\label{CPSK} Let $V$ be a finite dimensional rational $G$-module and let $m$ be a non-negative integer. Let $e$, $f$ be non-negative integers with $e\geq e(ctm)$, $f\geq f(c(\bar\lambda))$ for every weight $\lambda$ of $T$ in $V$. If $p\neq 2$ assume also $e\geq e(ct_p(\lambda)(m-1))+1$.

Then for $q=p^{e+f}$, the restriction map $\opH^n(G,V^{[e]})\to \opH^n(G(q),V)$ is an isomorphism for $n\leq m$ and an injection for $n=m+1$.\end{thm}

We  alert the reader that this result will be applied by first determining $e_0,f_0$ satisfying the inequalities required
of $e$ and $f$ above, respectively, and then checking $e\geq e_0$ and $f\geq f_0$, respectively, for actual
values of $e$ and $f$. which arise in our applications.

\begin{rems}\label{adjustment} (a) As pointed out in \cite[Remark 6.7(c)]{CPSK77}, it is not necessary to
check the numerical conditions in the theorem for each weight $\lambda$ of $V$. It is sufficient to check these
conditions for the highest weights of the composition factors of $V$.

(b) Let $e, f$ be as in Theorem \ref{CPSK}. For any $e'\geq e$, twisting induces a semilinear
isomorphism $\opH^m(G,V^{[e]})\overset\sim\to \opH^m(G,V^{[e']})$. In fact, the theorem implies there are
isomorphisms
\[\begin{CD}
\opH^m(G(p^{e'+f}),V^{[e]}) @<\sim<< \opH^m(G, V^{[e]}) @>\sim>> \opH^m(G(p^{e+f}),V^{[e]}),\end{CD}\]
where on the left-hand side we write $e'+f=e+f'$, $f'=f+e'-e$. Now consider the diagram (of semilinear
maps)
\[\begin{CD}
\opH^m(G, V^{[e']}) @>\sim>> \opH^m(G(p^{e'+f}), V^{[e']})\\
@AA^\text{twisting} A
\\
\opH^m(G,V^{[e]}) @>\sim>> \opH^m(G(p^{e+f}),V^{[e]})\end{CD}\]
The left-hand vertical map is an injection \cite{CPS83}, so since the two cohomology groups on the
right-hand side have the same dimension (notice $V^{[e']}\cong V^{[e]}\cong V$ as $G(q)$-modules
for any $q$), it must be an isomorphism. Compare \cite[Cor. 3.8]{CPSK77}.

In particular, $H^m(G,V^{[e]})\cong \opH_\gen^m(G,V)$ (essentially, by the definition of generic cohomology.)
\end{rems}

\begin{lem}\label{coarseCPSK}\begin{enumerate}\item[(a)]If $\lambda\in X^+_{r'}$, then $f(c(\bar\lambda))\leq r'+[\log_p(tc(2\rho)/2)]+2$. In particular, setting $f_0=f_0(\Phi)=\log_2(tc(2\rho)/2)+2$, we have $f(c(\bar\lambda))\leq r'+f_0$.
\item[(b)]Set $e_0=e_0(\Phi,m)=ctm$. Then, for $\lambda\in X^+$, $e_0\geq e(ctm)$ and, for all $p\neq 2$, $e_0\geq e(ct_p(\lambda)(m-1))+1$.\end{enumerate}\end{lem}

\begin{proof}(a) Note that
$c(\bar\lambda)= c(t\lambda)
\leq c(t(p^{r'}-1)\rho)
=t(p^{r'}-1)c(2\rho)/2$ and
$c(\bar\lambda)+1\leq tp^{r'}c(2\rho)/2$.
Thus, $f(c(\bar\lambda))=[\log_p(c(\bar\lambda)+1)]+2\leq  r'+[\log_p(tc(2\rho)/2)]+2$. The second statement is clear.

(b) Note that $t_p(\lambda)\leq t$. We leave the details to the reader.\end{proof}

As an overview for the proof of the following theorem, which becomes quite technical, let us outline the basic strategy. We show that if there is a non-trivial $m$-extension between two  $L(\lambda)$ and $L(\mu)$ which are  $q=p^r$-restricted, one can insist that $r$ is so big that Theorem \ref{digitsForG(q)} implies that the digits of the weights of the two irreducible modules must agree on a large contiguous string of zero digits. Since the cohomology for a finite Chevalley group is insensitive to twisting (as noted above), one can replace the modules with Frobenius twists. The resulting modules are still simple; so wrapping the resulting non-$p^r$-restricted  factors to the beginning we may assume they are $p^r$-restricted high weights $\lambda'$ and $\mu'$, respectively. In particular we can arrange that a large string of zero digits occurs at the end of $\lambda'$ and $\mu'$. This forces $\lambda'$ and $\mu'$ to be bounded away from $q$. The result is that we can apply Theorem \ref{CPSK} above.

  We recall some notation from the introduction. Let $q=p^r$ be a $p$-power. For $e\geq 0$ and $\lambda\in X$, there is a unique $\lambda'\in X^+_r$ such that $\lambda'|_{T(q)}=p^e\lambda|_{T(q)}$. Denote $\lambda'$ by
  $\lambda^{[e]_q}$.

\begin{thm}\label{shiftedGeneric} Let $\Phi$ be an irreducible root system and let $m\geq 0$ be given.

(a)
There exists a non-negative integer $r_0=r_0(\Phi,m)$ such that, for all $r\geq r_0$ and $q=p^r$ for any prime $p$, if $\lambda\in
X^+_r$, then, for some $e\geq 0$, there are  isomorphisms
$$\opH^n(G(q),L(\lambda))\cong\opH(G(q), L(\lambda)^{[e]})\cong \opH^n(G,L(\lambda')),\quad n\leq m,$$
where $\lambda'=\lambda^{[e]_q}$.  (The first isomorphism is semilinear.)  In addition, these isomorphisms can
be factored as
\[\opH^n(G(q),L(\lambda))\cong\opH^n(G(q),L(\lambda'))\cong \opH^n_\text{\gen}(G,L(\lambda'))\cong \opH^n(G,L(\lambda')).\]
Also, for any $\ell\geq 0$, the restriction maps
$$\opH^n(G(p^{r+\ell}),L(\lambda'))
\to \opH^n(G(q),L(\lambda'))\quad n\leq m$$
 are isomorphisms.  

(b) More generally, given a non-negative integer $\epsilon$, there is a constant $r_0= r_0(\Phi,m,\epsilon)
\geq \epsilon$
such that, for all $r\geq r_0$, if $\lambda\in X^+_r$ and $\mu\in X^+_\epsilon$, there exists an $e\geq 0$
and a semilinear isomorphism
$$\Ext^n_{G(q)}(L(\mu),L(\lambda))\cong \Ext^n_G(L(\mu'),L(\lambda')), \quad n\leq m,$$
where $\lambda'=\lambda^{[e]_q}$, $\mu'=\mu^{[e]_q}$. In addition,

 \begin{align*}
 \Ext^n_{G(q)}(L(\mu),L(\lambda)) &\cong \Ext^n_{G(q)}(L(\mu)^{[e]},L(\lambda)^{[e]})
  \cong \Ext^n_{G(q)}(L(\mu'),L(\lambda'))\\
  &\cong \Ext^n_{G,\gen}(L(\mu'),L(\lambda'))
 \cong  \Ext^n_{G}(L(\mu'),L(\lambda'))\\ \end{align*}
 for all $n\leq m$.  Also, for any $\ell\geq 0$, the restriction map $$\Ext^n_{G(p^{r+\ell})}(L(\mu'),L(\lambda'))
\to \Ext^n_{G(q)}(L(\mu'),L(\lambda'))$$ is an isomorphism.    
 \end{thm}

\begin{proof}Clearly, (a) is a special case of (b), so it suffices to prove (b).

By Theorem \ref{digitsForG(q)}, there is a constant $d=d(\Phi,n)$ so that, given $\lambda,\mu\in X^+_r$,
then $\lambda$ and $\mu$ differ in at most $d$ digits if $\Ext_{G(q)}^n(L(\mu),L(\lambda))\not=0$. 
We  may take $d(\Phi,n)\geq \delta(\Phi,n)$, where the latter constant satisfies Theorem \ref{digits}, so that $\lambda$
and $\mu$ differ in at most $d$ digits if $\Ext_G^n(L(\mu),L(\lambda))\not=0$.\footnote{For convenience, we are quoting
Theorems \ref{digitsForG(q)} and \ref{digits} with $\lambda$ and $\mu$ in reverse order. This does not
cause any problem.} Of course, we may also take $d\geq\epsilon$. If $e'>0$, these comments apply equally well to $p^{e'}\lambda$ and $p^{e'}\mu$.
 So
$\lambda$ and $\mu$, whose digits correspond naturally to those of $p^{e'}\lambda$ and $p^{e'}\mu$, differ in at most
$d$ digits if $\Ext^n_G(L(\mu)^{[e']}, L(\lambda)^{[e']})\not=0$. 
Therefore, if $\lambda$ and $\mu$ differ in more than $d=d(\Phi,n)$ digits, the claims of (b) hold.

In the same spirit, $\lambda$ and $\mu$ differ in at most $d$ digits, if $\Ext^n_{G(q')}(L(\mu),L(\lambda))
\not=0$ for some larger power $q'$ of $p$.
Otherwise, the relevant $\Ext$-groups all vanish, the isomorphism of (b) hold with
  $\lambda=\lambda'$, $\mu=\mu'$ and $e=0$.  

Put $d'=\max_{n\leq m} d(\Phi,n)$. By the discussion above, we can assume that  $\lambda$ and $\mu$ differ by at most $d'$ digits.  Recall also the constants $e_0=e_0(\Phi,m),f_0=f_0(\Phi)$ from Lemma \ref{coarseCPSK}. Set
$$r_0:=r_0(\Phi,m,\epsilon)=(d'+1)(e_0+f_0+g+1)+\epsilon-1,$$
where $g=g(\Phi):=[\log_2(h-1)]+2\geq [\log_p2(h-1)]+1$. We claim that (b) holds for $r_0$.
The hypothesis of (b) guarantees that $\mu\in X^+_\epsilon$, and $r\geq r_0$.  Observe
$r_0\geq\epsilon$, so that $\mu\in X_\epsilon^+$ is $p^r$-restricted. Also, $\lambda\in X^+_r$ under the
hypothesis of (b). In particular,  every composition factor $L(\tau)$ of $L(\mu^*)\otimes L(\lambda)$ is $p^{r+g}$-restricted by Lemma \ref{LenLem}.

For the remainder of the proof, we may assume that $\lambda$ and $\mu$ differ by at most $d'$ digits. 
By the ``digits" of $\lambda$ and $\mu$, we will mean just the first $r$ digits, the remainder being zero.
By hypothesis, $\mu\in X^+_\epsilon$, so all of its digits after the first $\epsilon$ digits are zero digits.
We claim that $\lambda$ and $\mu$ have a common string of at least $(r-\epsilon+1)/(d'+1)-1$ zero digits.
To see this, let $x$ denote the
longest string (common to both $\lambda$ and $\mu$) of zero digits after the first $\epsilon$ digits. Our
claim is equivalent (by arithmetic) to the assertion that 
 \begin{equation}\label{equivalent}
 (x+1)(d'+1)\geq r-\epsilon+1.\end{equation} 
 To see (\ref{equivalent}), call a digit position which is not zero for one of $\lambda$ or $\mu$ an
 exceptional position. Thus, 
every digit position past the first
$\epsilon$ positions is either one of the (at most) $d'$ exceptional digits, or occurs in a common string in $\lambda$ and $\mu$ of at
most $x$ zero digits, either right after an exceptional position, or right before the first exceptional position
(after the $\epsilon-1$ position). So, after the first $\epsilon-1$ positions,  there are
at most $d'+1$ strings of common zero digits, each of length at most $x$. Hence,  $(d'+1)x + d'\geq r-\epsilon$.
That is, $(d'+1)(x+1)\geq r-\epsilon+1$. This proves the inequality (\ref{equivalent}) and, thus, the claim.

Also, $x\geq (r-\epsilon+1)/(d'+1)-1\geq f_0+e_0+g_0$ since, by hypothesis,
$r\geq r_0=(d'+1)(f_0+e_0+g+1)+\epsilon-1$.

  We can take Frobenius twists $L(\lambda)^{[s]}$ and $L(\mu)^{[s]}$, with $s$ a non-negative
  integer, so that, up to the $r$th digit, the last $e_0+f_0+g$ digits of $\lambda^\circ:=\lambda^{[s]_q}$ and $\mu^\circ:=\mu^{[s]_q}$
  are zero.
   In particular, $\lambda^\circ$ and $\mu^\circ$ both belong to $X^+_{r-e_0-f_0-g}$. We are going to use for
  $e=e(\lambda,\mu)$ in (b) the integer $s+e_0$.
  We have
\[\Ext^n_{G(q)}(L(\mu^\circ),L(\lambda^\circ))\cong \opH^n(G(q),L(\mu^\circ)^*\otimes L(\lambda^\circ)),\]
and, by Lemma \ref{LenLem} again, the composition factors of
$M=L(\mu^\circ)^*\otimes L(\lambda^\circ)$ are in $X^+_{r-e_0-f_0-g+g}=X^+_{r-e_0-f_0}$. Let $L(\nu)$ be a composition factor of $M$. Then at least the last $e_0+f_0$ digits of $\nu$ are zero. Now, using Lemma \ref{coarseCPSK}, the weights of $M$ satisfy the hypotheses of Theorem \ref{CPSK}, and, thus, we have $\opH^n(G(q),M)\cong \opH^n(G,M^{[e]})$ for all $n\leq m$.  The same isomorphism holds if $q$ is replaced by any
larger power of $p$, so $\opH^n(G, M^{[e_0]})\cong\opH^n_\gen(G,M^{[e_0]})$.

Observe that $L(\lambda^{[e]_q})=L(\lambda^\circ)^{[e_0]}$, with a similar equation using $\mu$.
From the definition of $M$ above,
 $$\opH^n(G,M^{[e_0]})\cong \Ext_G^n(L(\lambda^\circ)^{[e_0]},L(\mu^\circ
)^{[e_0]})\cong\Ext^n_G(L(\lambda^{[e]_q}), L(\mu^{[e]_q})),\quad\forall n,  0\leq n\leq m,$$
and similar isomorphisms hold for $G(q)$-cohomology and $\Ext^n_{G(q)}$-groups.
We now have most of the isomorphisms needed in (b), with the remaining ones
obtained from group automorphisms on $G(q)$.

 This completes the proof.
 \end{proof}

\begin{rem}\label{needATwist}
\cite[Thm. 5.6]{BNP06} shows that when $m=1$,  $r\geq 3$, $p^r\geq h$, then, with  $e=[(r-1)/2]$, we have $\Ext^1_{G(q)}(L(\lambda),L(\mu))\cong \Ext^1_{G}(L(\lambda)^{[e]_q},L(\mu)^{[e]_q})$.

It is tempting to think, as suggested in \cite[Question 3.8]{SteSL3}, that one might have similar behavior for higher values of $m$ for some integer $e\geq 0$ under reasonable conditions. Unfortunately, for $p$ sufficiently large, this is never true:

In \cite[Thm. 1]{SteUnb} the third author gives an example of a module $L_{r-1}$ for $SL_2$ over $\bar F_p$ with $p>2$ with the property that the dimension of $\Ext^2_G(L_{r-1},L_{r-1})=r-1$. Specifically, $L_r$ is $L(1)\otimes L(1)^{[1]}\otimes\dots \otimes L(1)^{[r]}$. So $L_{r-1}$ is $q=p^r$-restricted; it is self-dual, since this is true for all simple $SL_2$-modules; it also has the property that $L_{r-1}^{[e]_q}=L_{r-1}$ for any $e\in\N$. Note that as $L(2)$ is isomorphic to the adjoint module for $p>2$, we have $\dim \opH^2(G(q),L(2)^{[i]})\geq 1$ for any $i\geq 0$.
Set $D=\dim\Ext^2_{G(q)}(L_{r-1},L_{r-1})$. We show $D>\dim\Ext^2_G(L_{r-1},L_{r-1})=r-1$.

We have \begin{align*}D&\cong \Ext^2_{G(q)}(k,L_{r-1}\otimes L_{r-1})\\
&\cong \opH^2(G(q), L_{r-1}\otimes L_{r-1})\\
&\cong \opH^2(G(q), (L(1)\otimes L(1))\otimes (L(1)\otimes L(1))^{[1]}\otimes\dots\otimes (L(1)\otimes L(1))^{[r-1]})\\
&\cong \opH^2(G(q), (L(2)\oplus L(0))\otimes (L(2)\oplus L(0))^{[1]}\otimes \dots\otimes(L(2)\oplus L(0))^{[r-1]})\\
&\geq \opH^2(G(q), L(2)\oplus L(2)^{[1]}\oplus\dots\oplus L(2)^{[r-1]})\\
&\geq r>r-1,\end{align*}
as required.

Indeed, \cite[Rem. 1.2]{SteUnb} gives a recipe for cooking up such examples for simple algebraic groups having any root system---one simply requires $p$ large compared to $h$.

Essentially the problem as found above can be described by saying that $\Ext^2_G(L_{r-1},L_{r-1})$ is not rationally stable. One does indeed have $\dim \Ext^2_G(L_{r-1}^{[1]},L_{r-1}^{[1]})=D$.\end{rem}

Motivated by the above example, we ask the following question, a modification of \cite[Question 3.8]{SteSL3}
\begin{ques}\label{conjecture} Let $e_0=e_0(\Phi,m):=ctm$. Does there exist a constant $r_0=r_0(\Phi,m)$, such that for all $r\geq r_0$, the following holds:

 For $q=p^r$, if $\lambda,\mu\in X^+_r$, then there exists a non-negative integer
 $e=e(\lambda,\mu)$ such that

 \begin{align*}
 \Ext^n_{G(q)}(L(\lambda),L(\mu)) &\cong \Ext^n_{G(q)}(L(\lambda)^{[e]},L(\mu)^{[e]})
  \cong \Ext^n_{G(q)}(L(\lambda^{[e]_q}),L(\mu^{[e]_q}))\\
  &\cong \Ext^n_{G,\text{\rm gen}}(L(\lambda^{[e]_q}),L(\mu^{[e]_q}))
 \cong  \Ext^n_{G}(L(\lambda^{[e]_q})^{[e_0]},L(\mu^{[e]_q})^{[e_0]})\\ \end{align*}
 for $n\leq m$?\end{ques}

\begin{rem}We make the simple observation that in the theorems above one needs the weights $\lambda$ and $\mu$ to be $p^r$-restricted, hence, these weights determine simple modules for $G(q)$. For instance, with the notation of the previous remark, again with $G=SL_2$ and $p>2$, $\opH^0(G,L_{2n-1})=0$, but \[\opH^0(G(p),L_{2n-1})\cong\opH^0(G(p),(L(2)+L(0))^{\otimes n})\not=k,\]
with a similar phenomenon occurring for larger values of $q$. So even the $0$-degree cohomology of $G$ will not agree with that of $G(q)$ on general simple $G$-modules.\end{rem}

Using results of \cite{BNP01}, we draw the following striking corollary from the main result of this section
(and paper). Let $W_p=W\ltimes p{\mathbb Z}\Phi$ be the affine Weyl group and
$\widetilde W_p=W\ltimes p{\mathbb Z}X$ be the extended affine Weyl group for $G$. Both groups
 act on
$X$ by the ``dot" action: $w\cdot\lambda= w(\lambda+\rho)-\rho$. 

\begin{thm}\label{summary} For a given non-negative integer $m$ and irreducible root system $\Phi$,  there is, for
all but finitely many prime powers $q=p^r$, an isomorphism
\begin{equation}\label{iso}\opH^m(G(q), L(\mu))\cong \opH^m(G,L(\mu')), \quad \mu\in X_r^+,\end{equation}
for some constructively given dominant weight $\mu'$. If $r$ is sufficiently large, we can take
$\mu'\in X^+_r$ to be a $q$-shift of $\mu$.  If $p$ is sufficiently large, and if $\mu$ is $\widetilde W_p$-conjugate to 0, then we can take $\mu'=\mu$.)

In particular, there is a bound $C=C(\Phi,m)$ such that $\dim \opH^m(G(q), L(\mu))\leq C$ for all values of $q$ and $q$-restricted weights $\mu$.\end{thm}

\begin{proof} We first prove the assertions in the first paragraph. It suffices to treat the case $m>0$. If $r$ is sufficiently large, then (5.4) holds by Theorem \ref{shiftedGeneric} for some $q$-shift
$\mu'\in X^+_r$ of $\mu$. On the other hand, suppose that $p\geq (4m+1)(h-1)$.    We can assume that $\mu$ is $p$-regular, otherwise \cite[Cor. 7.4]{BNP01} tells us that
$\opH^m(G(q),L(\mu))\cong\opH^m(G,L(\mu))=0$, and there is nothing to prove. 

Suppose $\mu=u\cdot \nu$ for some $\nu\in X^+_r$ satisfying $(\nu,\alpha_0^\vee)\leq 2m(h-1)$, and
some $u\in \widetilde W_p$. Then
\cite[Thm. 7.5]{BNP01} states that
$$\opH^m(G(q),L(\mu))\cong\opH^m(G,L(u\cdot 0+p^r\nu)),$$
so that (\ref{iso}) holds in this case. If $\mu$ does not have the form $\mu=u\cdot \nu$ as above, set
$\mu'=\mu$. The first paragraph of the proof of \cite[Thm. 7.5]{BNP01} shows that $\opH^m(G(q),L(\mu))=0$.
Also, $\opH^m(G,L(\mu'))=\opH^m(G,L(\mu))=0$, by the Linkage Principle, since $\mu$ is not $W_p$-linked to 0. 

It remains to prove the statement in the second paragraph. We have just established that there is a number $q_0$ such that for all prime powers $q=p^r\geq q_0$,  for any $\mu\in X^+_r$, there exists a $\mu'\in X^+$
such that $\opH^m(G(q), L(\mu))\cong \opH^m(G,L(\mu'))$. By \cite[Thm. 7.1]{PS11}, the 
numbers $\dim \opH^m(G,L(\mu'))$ are bounded by a constant $c=c(\Phi,m)$ depending only on $m$ and $\Phi$. Let $c'=\max\{\dim \opH^m(G(q), L(\mu)\}$, the maximum taken over all prime powers $q=p^r<q_0$ and weights $\mu\in X_r^+$; clearly $c'$ is finite. Then  $\dim \opH^m(G(q), L(\mu))\leq \max\{c',c\}$. 
 \end{proof}

 The explicit bounds exist for $r$ to be sufficiently large, see Theorem \ref{shiftedGeneric}. The explicit bound on $p$ in the proof can be improved using
 Theorem \ref{thm6.2}(c). The constructive dependence of $\mu'$ on $\mu$ merely involves the combinatorics
 of weights and roots.
 
We can give the following corollary, addressed more thoroughly in \S6 below; see Theorem \ref{thm6.2}(c) and Theorem \ref{lastcor}.
 
 \begin{cor} If $p$ is sufficiently large, depending on $\Phi$ and the non-negative integer $m$, every weight
 of the form $p\tau$, $\tau\in X^+$, is $m$-generic at $q$, any power $q$ of $p$ for which $p\tau$ is $q$-restricted.
 In addition, if $\mu\in X^+$ is $q$-restricted, and has a zero digit in its $p$-adic expansion
 $\mu=\mu_0+p\mu_1+\cdots+p^{r-1}\mu_{r-1}$ ($p^r=q$), then $\mu$ is shifted $m$-generic at $q$. Moreover,
 in the first case, 
 $$\opH^m(G(q),L(p\tau))\cong\opH^m(G(q),L(\tau))\cong \opH^m_{\gen}(G),L(\tau))\cong \opH_\gen^m(G,L(p\tau)),$$ 
 and, in the second case,
 $$\opH^m(G(q),L(\mu'))\cong\opH^m(G(q),L(\mu))\cong \opH^m_{\gen}(G,L(\mu))\cong \opH_\gen^m(G,L(\mu')).$$ 
 \end{cor}
 \begin{proof} For the first part, observe that $p\tau$ is $\widetilde W_p$-conjugate to 0. So the first part
 follows from Theorem \ref{summary}. For the second part, simply observe that if $\mu$ has a zero digit
(among its first $r$ digits), there is some $q$-shift $\mu'$ of $\mu$ with $\mu'=p\tau$ for some
dominant $\tau$. 
 \end{proof}
 
 \section{Large prime results}   In this section, we give some ``large prime" results. Much work has
 been done on this topic, see \cite{BNP01} and \cite{BNP02}, as well as earlier papers \cite{HHA84} and
 \cite{FP83}. 

The following result for $p\geq 3h-3$ is given in \cite[Cor. 2.4]{BNP02}.

\begin{thm}\label{Ext1digits} Assume $p\geq h$ and let $\lambda,\mu\in X^+_r$

(a) Suppose  $\Ext^1_G(L(\lambda),
L(\mu))\not=0$. Then $\lambda$ and $\mu$ differ in at most two digits, which must be adjacent.

(b) Let $q=p^r$. If $\Ext^1_{G(q)}(L(\lambda),L(\mu))\not=0$, then $\lambda$ and $\mu$ differ in at most 2 digits, which must be either adjacent, or the first and the last digits. \end{thm}

\begin{proof} For (a), we just need to apply Remark \ref{special} since $2+[\log_p(h-1)]=2$ in this case.
Part (b) follows from (a) and \cite[Thm. 5.6]{BNP06}. (Note that both (a) and (b) are trivial unless $r\geq 3$.)
 \end{proof}

The following result combines $q$-shifting with the themes of \cite[Thms. 7.5, 7.6]{BNP01}. There is some
overlap with the latter theorem, which is an analogue for $G(q)$ and all $m$ of Andersen's well-known
$m=1$ formula for $\Ext^1_G$ \cite[Thm. 5.6]{HHA84}. However,  we make no assumption, unlike \cite[Thm. 7.6]{BNP01}, that $\mu$ be sufficiently far from the walls of its
alcoves. For the $m=1$ case and $r\geq 3$, \cite[Thm. 5.6]{BNP02} gives a much sharper formula, as 
noted in Remark \ref{needATwist}.

\begin{thm}\label{thm6.2} (a) Assume that $p\geq 6m+7h-9$. Then for $q=p^r$ (any $r$), and $\lambda,\mu\in X^+_r$, 

\begin{equation}\label{thm6.2a}\Ext^m_{G(q)}(L(\lambda),L(\mu))\cong\bigoplus_{\nu}\Ext^m_G(L(\lambda)\otimes L(\nu),L(\mu)\otimes L(\nu)^{[r]}),
\end{equation}
where $\nu\in X^+$ runs over the dominant weights in the closure of the lowest $p$-alcove.  

(b)  Assume that $p> 12m +13h-16$, and $\lambda,\mu\in X^+_r$ with $\lambda$ having a zero digit.
Then  $\lambda, \mu$ can be replaced, maintaining the
dimension of the left-hand side of (\ref{thm6.2a}), by suitable simultaneous $q$-shifts $\lambda',\mu'$ so that the sum on the right-hand side of (\ref{thm6.2a}) collapses to a single summand,
\begin{equation}\label{collapse}
\Ext^m_{G(q)}(L(\lambda),L(\mu))\cong\Ext^m_{G(q)}(L(\lambda'),L(\mu'))\cong
\Ext^m_G(L(\lambda')\otimes L(\nu),L(\mu')\otimes L(\nu^*)^{[r]})\end{equation}
for some (constructively determined) $\nu$ in the lowest $p$-alcove. Also, $\lambda'$ can be chosen 
to be any $q$-shift whose first digit is zero (possibly with different weights $\nu$ for different choices
of $\lambda')$. In addition, $L(\lambda')\otimes L(\nu)$ and $L(\mu')\otimes L(\nu^*)^{[r]}$ are irreducible
in this case.

(c) Moreover, again with $p> 12m + 13h -16$ as in (b), assume that $\lambda,\mu\in X^+_r$ have a
common 0 digit (among the first $r$ digits, with $q=p^r$). Then simultaneous\footnote{See
the discussion above Theorem \ref{lastcor} for the precise definition of a simultaneous $q$-shift.} $q$-shifts $\lambda',\mu'$ may be chosen so that
\begin{equation}\label{collapse2}
\Ext^m_{G(q)}(L(\lambda),L(\mu))\cong\Ext^m_{G(q)}(L(\lambda'),L(\mu'))\cong\Ext^m_{G,\gen}(L(\lambda'),L(\mu'))
\cong \Ext^m_G(L(\lambda'),L(\mu')).\end{equation}
These isomorphisms all hold, for any simultaneous $q$-shifts $\lambda',\mu'$ of $\lambda,\mu$ 
for which $\lambda',\mu'$ both have a zero first digit.
  \end{thm}

\begin{proof}We first prove (a).  Let $p\geq 6m+7h-9$. Put $b:=6m+6h-8$. Then if $\nabla(\nu)\otimes\nabla(\nu^*)^{[r]}$
is a section in $\sG_{r,b}$, we have 
\begin{equation}\label{smallpoint} (\nu+\rho,\alpha_0^\vee)\leq b + (\rho,\alpha_0^\vee)=
b+ h-1=6m+7h-9\leq p.\end{equation}
 Therefore, $\sG_{r,b}$ is completely reducible as a rational $G$-module
with summands $L(\nu)\otimes L(\nu^*)^{[r]}$,  in which $\nu\in X^+$ is in the closure of the lowest $p$-alcove.\footnote{For larger $p$, $\nu$ is actually in the interior of the lowest $p$-alcove. In particular, this applies
to parts (b) and (c) of the theorem.} 
(Of course, $\nabla(\nu)=L(\nu)$ for $\nu$ in the closure of the lowest $p$-alcove.) 
For $p>6m+7h-9$, there are dominant weights $\nu$ in the closure of the lowest $p$-alcove
which do not satisfy $(\nu,\alpha_0^\vee)\leq b.$ 
For such $\nu$, (\ref{vanishing}) gives that $\Ext^m_G(L(\lambda),L(\mu)\otimes L(\nu)\otimes L(\nu^*)^{[r]})=0$. Thus, (a) follows. 

To prove (b), assume that $p>12m + 13h-16= 2b + h$. Choose $e$ with $0\leq e<r$, so that $\lambda':=\lambda^{[e]_q}$ has its first digit equal to
zero. Put $\mu':=\mu^{[e]_q}$. Then the left-hand isomorphism in (\ref{collapse}) holds. In addition, 
the isomorphism (\ref{thm6.2a}) holds. 

There is an expression like (\ref{thm6.2a}) with $\lambda,\mu$ replaced by $\lambda',\mu'$.  If one of the terms indexed by $\nu$ on the right-hand side of (\ref{thm6.2a}) (for $\lambda',\mu'$) is nonzero,  then $\mu'$ is $\widetilde W_p$-conjugate to $\nu$. To see this,
first note that
$\lambda'=p\lambda^\dagger$ for some $\lambda^\dagger\in X^+$, so that $L(\lambda')\otimes L(\nu)$ is
irreducibile by the Steinberg tensor product theorem. Similarly, $L(\mu)\otimes L(\nu^*)^{[r]}$ is irreducible
since $\mu\in X^+_r$. If $\Ext^m_G(L(\lambda')\otimes L(\nu),L(\mu')\otimes L(\nu^*)^{[r]})\not=0$, then $\nu+p\lambda^\dagger$ is $W_p$-conjugate to $\mu'+p^r\nu^*$. Therefore, $\nu$ and
$\mu'$ are $\widetilde W_p$-conjugate.  Now it is only necessary to show that any two dominant
weights that are $b$-small and $\widetilde W_p$-conjugate are equal. Briefly, suppose that $\nu,\nu'$
are $\widetilde W_p$-dot conjugate dominant weights in the lowest $p$-alcove. Write
$\nu'=w.\cdot\nu+ p\tau$, for $w\in W$, $\tau\in X$. If $\tau=0$, then $\nu=\nu'$ because both weights are dominant. Hence, $\tau\not=0$. Thus, for $\alpha\in\Pi$, $|(\nu'-w\cdot\nu,\alpha^\vee)|
\leq |(\nu'+\rho,\alpha^\vee)| + |(w(\nu+\rho),\alpha^\vee)| \leq 2b+h$ by the first part of (\ref{smallpoint}). 
But $2b+ h< p$. 
  Choose $\alpha\in\Pi$ such that $(\tau,\alpha^\vee)\not=0$.
It follows then that $|(\tau,\alpha^\vee)|=1$, an evident contradiction. 
 This completes the proof of (b).\footnote{Apart from the use of $q$-shifts and different
numerical bounds, the argument in this paragraph parallels that of \cite[Thm. 7.5]{BNP01}.}$^,$
\footnote
{Conceptually, the lowest $p$-alcove $C:=\{\lambda\in {\mathbb R}\otimes X\,|\, 0\leq(\lambda+\rho,\alpha_0^\vee)<p\}$ is the union of closed simplices conjugate under the finite
subgroup $N$ of $\widetilde W_p$ stabilizing $C$. The group $N$ acts transitively and regularly on the interiors of these simplices. With our assumptions on the sizes of the various $(\nu,\alpha_0^\vee)$,
the relevant dominant $\nu$ all belong to the interior of the ``lowest" simplex---the one containing 0.}

Finally, to prove (c), we continue the proof of (b), taking $\mu'$ to also have a zero first digit. Put $\mu'=p\mu^\dagger$, for some $\mu^\dagger\in X^+$. So $\mu'$, and thus $\nu$ above, must be $\widetilde W_p$-conjugate to 0. Therefore $\nu=0$, giving all the isomorphisms in (\ref{collapse2}) except the ones involving $\Ext^m_{G,\gen}$. We observe that the $\widetilde W_p$-conjugacy of $\nu$ to $0$ still holds passing 
from $r$ to $r+e_0$ for any integer $e_0\geq 0$. Thus, $\Ext^m_{G(q)}(L(\lambda'),L(\mu'))
\cong \Ext^m_{G}(L(\lambda'),L(\mu'))$ and $\Ext^m_{G(p^{r+e_0})}(L(\lambda'),L(\mu'))
\cong\Ext^m_G(L(\lambda'), L(\mu'))$,
 so $$\Ext^m_G(L(\lambda'),L(\mu'))\cong\Ext^m_{G,\gen}(L(\lambda'),L(\mu')).$$
  This proves (c).
\end{proof}

\begin{rems}\label{rem6.3} (a) In the situation of (b), suppose it is $\lambda$ itself that has first digit 0, and put $\lambda'=\lambda$ and $\mu'=\mu$. Assume that $\Ext^m_G(L(\lambda),L(\mu))\not=0$. Then the conclusion of (c) holds. To see this, note that the non-vanishing implies $\lambda$ and $\mu$ are $W_p$-conjugate, and
thus $\widetilde W_p$-conjugate. Now, we can use the proofs of (b) and (c).

(b) The requirements in (b) and (c) on the existence of a simultaneous zero cannot be dropped. An example is provided in Remark \ref{needATwist}.

(c) The sum on the right-hand side of (\ref{thm6.2a}) always involves a bounded number of
summands, independent of $p$, namely, those in which $(\nu,\alpha_0^\vee)\leq 6m+6h-8$. 

\end{rems}

 Taking $\nu=0$ in  (\ref{thm6.2a}) yields the following result.
 
\begin{cor}\label{cor6.4} If $p\geq 6m+7h-9$, then restriction $\Ext^m_{G}(L(\lambda),L(\mu))\to \Ext^m_{G(q)}( L(\lambda),
L(\mu))$
is an injection for every $\lambda,\mu\in X^+_r$. \end{cor}

Suppose that $\lambda$ and $\mu$ are $q$-restricted weights. Let us call a pair $(\lambda',\mu')$ a $q$-shift of $(\lambda,\mu)$ if it is obtained by a 
simultaneous $q$-shift $\lambda'=\lambda^{[e]_q}$, $\mu'=\mu^{[e]_q}$.  Also, we say the pair $(\lambda,\mu)$ of $q$-restricted
weights is $m$-generic at $q$ if
 $$\Ext^m_{G(q)}(L(\lambda),L(\mu))\cong\Ext^m_G(L(\lambda),L(\mu)).$$
Similarly, we say the pair $(\lambda,\mu)$ is shifted $m$-generic if $(\lambda',\mu')$ is  $m$-generic
at $q$ for a $q$-shift $(\lambda',\mu')$ of $(\lambda,\mu)$. Theorem \ref{thm6.2}(c) asserts that for
$p$ large, pairs $(\lambda,\mu)$ are shifted $m$-generic at $q$. We improve this in the case that
a zero digit occurs as the last digit. At the same time, we give an improvement, in the large prime case, to
the limiting procedure of \cite{CPSK77} in the result below. Part (a) is already implicit in \cite{CPSK77}
with a different bound, but (b) and (c) are new and at least theoretically interesting,  since they give
the best possible value for the increase required in $q$ to obtain stability; see the examples in
Remark \ref{lastremark} which follows.

\begin{thm}\label{lastcor} 
Assume that $p>12m + 13h-16$, and let $\lambda,\mu\in X^+_r$. Let $q=p^r$.

(a) The semilinear map 
$$\Ext^m_G(L(\lambda)^{[1]},L(\mu)^{[1]})\to\Ext^m_G(L(\lambda)^{[e]},L(\mu)^{[e]})$$
is an isomorphism for every $e\geq 1$. In particular, 
$$\Ext^m_{G,\gen}(L(\lambda),L(\mu))\cong \Ext^m_G(L(\lambda)^{[1]},L(\mu)^{[1]}).$$

(b)
Put 
$q'=p^{r+1}$.  Then the $q'$-shifts $\lambda'=\lambda^{[1]_{q'}}$, and $\mu'=\mu^{[1]_{q'}}$ satisfy
$$\Ext^m_{G(q')}(L(\lambda),L(\mu))\cong\Ext^m_{G(q')}(L(\lambda'),L(\mu'))\cong\Ext^m_{G,\gen}(L(\lambda'),L(\mu'))\cong\Ext^m_G(L(\lambda'),L(\mu')).$$
In addition, for any $s\geq 1$, 
the map 
$$\Ext^m_{G(p^{r+s})}(L(\lambda'),L(\mu'))\to\Ext^m_{G(q')}(L(\lambda'),L(\mu'))$$
is an isomorphism as is 
$$\Ext^m_{G(p^{r+s})}(L(\lambda),L(\mu))\to\Ext^m_{G(q')}(L(\lambda),L(\mu)).$$
In particular, the pair $(\lambda,\mu)$ is always $m$-generic at $q'=p^{r+1}$.

(c)  Let $M,N$ be finite dimensional rational $G$-modules whose composition factors 
are all $p^r$-restricted. 
Then, if $s\geq 1$, the natural restriction map 
$$\Ext^n_{G,\gen}(M,N)\to \Ext^n_{G(p^{r+s})}(M,N)$$
is an isomorphism for $n\leq m$ and an injection for $n=m+1$.
\end{thm}

\begin{proof} We begin by remarking that (c) follows from (a) and (b): Note that Corollary \ref{cor6.4} applies with $m$ replaced by $m+1$,
checking the required condition on $p$. Applied to $(p\lambda,p\mu)$ and assuming (a), Corollary \ref{cor6.4}
gives an injection $$\Ext^{m+1}_{G,\gen}(L(p\lambda),L(p\mu))\to\Ext^{m+1}_{G(q')}(L(p\lambda),L(p\mu)).$$
It follows that there  is a corresponding injection with $(p\lambda,p\mu)$ are replaced by $(\lambda,\mu)$. Now (c)
follows from this latter injection and 
the last isomorphism in (b), valid also with $m$ replaced by
 any smaller integer. (This is a well-known 5-lemma argument needing only the
injectivity for the degree $m+1$-maps.) 

So it remains to prove (a) and (b).  The first display in (b) follows from Theorem 6.2(c). We get
a similar string of isomorphisms with $q'$ replaced by $p^{r+s}$. Note that $\lambda^{[1]_{p^{r+s}}}
=\lambda^{[1]_{q'}}=\lambda'$, with a similar equation for $\mu'$. (We use here the fact that $\lambda,\mu$
are $q$-restricted.) Now consider the the following commutative diagram, where the vertical maps
are restriction maps:
\[
\begin{CD}
\Ext^m_{G(p^{r+s})}(L(\lambda),L(\mu)) @>\sim>> \Ext^m_{G(p^{r+s})}(L(\lambda'),L(\mu'))@>\sim>>
\Ext^m_{G,\gen}(L(\lambda),L(\mu))\\
@VVV @VVV @|\\
\Ext^m_{G(q')}(L(\lambda),L(\mu)) @>\sim>>  \Ext^m_{G(q')}(L(\lambda'),L(\mu'))@>\sim>>\Ext^m_{G,\gen}(L(\lambda),L(\mu)) 
\end{CD}
\]
It follows easily that two vertical maps (on the left) are isomorphisms, as required. This proves (b). 

To prove (a), note that $\dim\Ext^m_G(L(\lambda)^{[1]},L(\mu)^{[1]})=\dim\Ext^m_{G,\gen}(L(\lambda'),L(\mu'))$ (by the first isomorphism in (b)) which equals $\dim\Ext^m_{G,\gen}(L(\lambda),L(\mu))$ by
definition. However, the map
$$\Ext^m_G(L(\lambda)^{[1]}, L(\mu)^{[1]})\to\Ext^m_{G,\gen}(L(\lambda),L(\mu))$$
is injective by \cite{Dect}. So,
by dimension considerations, it must be an isomorphism. Part (a) follows easily.
\end{proof}
\begin{rem}\label{lastremark} For an example of a pair $(\lambda,\mu)$ of $q$-restricted weights that is not $m$-generic or shifted $m$-generic at $q$ with
$m=2$, see Remark \ref{needATwist}. Even when $p$ is large, there are examples for fixed $\Phi$ (of
type $A_1$) and fixed $m$ ($=2$) for arbitrarily large $r$. Of course, these examples have no zero digits
in common.
\end{rem}
 
 \medskip\medskip
 \begin{center} {\bf Appendix} \end{center}
 
 \medskip
 
 In this brief appendix, we consider the large prime generic cohomology results of \cite[\S3]{FP86} from the
 point of view of Theorem \ref{lastcor} and the other results of this paper. 
  
  Suppose $q=p$ is prime and let $\mu\in X^+_1$. Assume $p> 12m + 13h- 16=2b+h$, as in Theorem \ref{thm6.2}(b). Taking 
$\lambda=0$, we thus get
\begin{equation}\label{appendixequal}
\opH^m(G(p), L(\mu))\cong\Ext^m_G(L(\nu),L(\mu)\otimes L(\nu^*)^{[1]}),\end{equation}
for some $b$-small $\nu\in X^+$. 

Necessarily $\mu+p\nu^*$ lies in the Jantzen region. If $\mu+p\nu^*\not\in W_p\cdot\mu$, then (\ref{appendixequal})
vanishes. Otherwise, take $p$ larger, if necessary, so 
  that the Lusztig character formula holds for $G$. Then the dimension of (\ref{appendixequal}) equals the coefficient of a  Kazhdan-Lusztig polynomial, since $L(\nu)\cong \Delta(\nu)$; 
see \cite[\S3]{CPS93}.\footnote{As noted in \cite[Thm. 7.5]{BNP01} and its proof, which also give
the above isomorphism for large $p$, the right-hand side can be converted to a cohomology group (of an
irreducible module)
via a translation taking $\nu$ to $0$.  Thus, the relevant Kazhdan-Lusztig polynomial has the form $P_{w_0,w_0w}$, where 
$l(w_0w)=l(w_0) + l(w)$, $w\in W_p$.}
The Lusztig
character formula is known to  hold when $p\gg h$ \cite{AJS94}, and is conjectured to be true
for $p\geq h$.  
 
 In \cite[Prop. 3.2]{FP86}, it was observed that $\opH^m(G(p), L(\mu))\cong H^m_\gen(G,L(\mu))$ 
 if $p$ is sufficiently large, depending on the highest weight $\mu$ and the integer $m$. More precisely, in
 this case, taking $\mu$ to lie in the closure of the  lowest $p$-alcove,
 \cite[Thm. 3.3, Cor. 3.4]{FP86} gives a dimension formula, valid when $p$ is large,
 \begin{equation}\label{formula}
 \dim\opH^m(G(p),L(\mu))=\begin{cases} 0\,\,\, m\,{\text{\rm odd},}\\
 \sum_{w\in W} \det(w){\mathfrak p}_{m/2}(w\cdot\mu)\,\,\, m\,\,{\text{\rm even}}.\end{cases}\end{equation}
 Here $\mathfrak p$ is the Kostant partition function. It is interesting to rederive this result in our
 present context, and compare it with (\ref{appendixequal}).
 
 In addition to assuming that $p> 2b+h$, also assume that $p\geq(\mu,\alpha_0^\vee)
 +h-1$.
 (The last condition just says $\mu$ lies in the closure
  of the lowest $p$-alcove, so that $L(\mu)\cong\nabla(\mu)$.) Theorem \ref{lastcor}(a) shows that
  $\opH^m_\gen(G,L(\mu))\cong\opH^m(G,L(\mu)^{[1]})\cong\opH^m(G,\nabla(\mu)^{[1]})$. A formula
  for the dimension of the latter module is derived in \cite[Prop. 4.2]{CPS09}, which
  is precisely that given in the right-hand side of (\ref{formula}).
  
  Under the same conditions on $p$ as in the previous paragraph, but perhaps enlarging $p$ further,  we claim there is an identification of $\opH^m_\gen(G,L(\mu))$ 
 with $\opH^m(G(p^r),L(\mu))$, for all positive integers $r$. For $r\geq 2$, this claim follows from Theorem \ref{lastcor}(c). 
 For the case $r=1$, we also require $\mu$ to be $b$-small, a condition on $p$ when $\mu$ is fixed.
 
Return now to (\ref{appendixequal}) in the case of a $b$-small $\mu\in X^+$.  By the argument for Theorem \ref{thm6.2}(b), there is no other $b$-small dominant weight $\widetilde W_p$-conjugate of $\mu$. Thus, we can assume $\nu=\mu$. Next, consider
 $\opH^m_\gen(G,L(\mu))\cong\opH^m(G,L(\mu)^{[1]})$ by Theorem \ref{lastcor}(a). If this generic cohomology
 is non-zero, then $\mu$ belongs to the root lattice. If translation to the principal 
 block is applied to
 $L(\mu)\otimes L(\mu^*)^{[1]}$, using \cite[Lemma 3.1]{CPS09}, we obtain an irreducible
  module $L(\tau)\otimes L(\mu^*)^{[1]}$ with $\tau$  in 
   lowest $p$-alcove, and with highest weight also in $W_p\cdot 0$. 
  Therefore, $\tau=0$.  Clearly, translation to the principal block takes $L(\mu)$ to $L(0)$. Thus,
 $$\begin{aligned} \opH^m(G(p),L(\mu))& \cong\Ext^m_G(L(\mu),L(\mu)\otimes L(\mu^*)^{[1]})\\
 &\cong\Ext^m_G(L(0),L(\mu^*)^{[1]})\\
& \cong \opH^m(G,L(\mu^*)^{[1]})\\
& \cong \opH^m(G, L(\mu)^{[1]})\\ &\cong \opH^m_\gen(G, L(\mu))\end{aligned}
$$ in this case. If $\opH^m_\gen(G,L(\mu))=0$, we claim that $H^m(G(p),L(\mu))=0$ also. Otherwise,
$$\Ext^m_G(L(\mu),L(\mu)\otimes L(\mu^*)^{[1]})\not=0.$$ Thus, $\mu$ and $\mu+p\mu^*$ belong to the
same $W_p$-orbit, forcing $\mu^*$ to belong to the root lattice, again giving (by translation arguments) $\opH^m(G(p),L(\mu))=\opH^m_\gen(G,L(\mu))=0$, a contradiction. This completes the proof of the claimed
identification.

Finally, observe the answer we  obtained for $\dim\opH^m(G(p),L(\mu))$, for our rederivation of 
(\ref{formula}),
is, when $\mu$  is $b$-small
and lies in
the root lattice,  a Kazhdan-Lusztig
polynomial coefficient. (The dimension is independent of $p$ and
has the form $\dim\Ext^m_G(\Delta(0),L(p\mu))$ with $p\mu$ in the Jantzen region.)  In this case,
the Kazhdan-Lusztig polynomial coefficient that gives the right-hand side of (\ref{appendixequal}) is associated
to $\Ext^m_G(\Delta(0),L(p\mu^*))$, which is the same coefficient. (Apply a graph automorphism.) When
$p\mu$ is not in the root lattice, $\mu+p\mu^*$ is not in $W_p\cdot\mu$, and so the right-hand side of (\ref{appendixequal}) is zero, as is $\dim\Ext^m_G(\Delta(0), L(p\mu))$. Consequently, the combinatorial
determination of (\ref{appendixequal}) is, for $b$-small $\mu$, the same  as that for (\ref{formula}), if Kazhdan-Lusztig polynomial coefficients are used in both cases.

But the determination of the dimension of (\ref{appendixequal}) as a Kazhdan-Lusztig polynomial coefficient or zero applies for all restricted
$\mu$, if $p\gg0$, not just for those that are $b$-small or lie in the closure of the bottom $p$-alcove. 
Thus, in some sense, the discussion above may be viewed as giving a generalization of the determination in \cite{FP86} of
$\opH^m(G(p),L(\mu))$ for $p$ sufficiently large, depending on $m$ and $\mu$.

{\footnotesize
\bibliographystyle{alpha}
\bibliography{PSbib.bib}}

\begin{thebibliography}{CPSvdK77}

\bibitem[AJL83]{AJL83}
Henning Andersen, Jens J{\o}rgensen, and Peter Landrock.
\newblock The projective indecomposable modules of {${\rm SL}(2,\,p^{n})$}.
\newblock {\em Proc. London Math. Soc. (3)}, 46(1):38--52, 1983.

\bibitem[AJS94]{AJS94}
Henning Andersen, Jens Jantzen, and Wolfgang Soergel.
\newblock {\em Representations of quantum groups at a $p$th root of unity and
  of semisimple groups in characteristic $p$}, volume 220.
\newblock Ast\'erique, 1994.

\bibitem[And84]{HHA84}
Henning Andersen.
\newblock Extensions of modules for algebraic groups.
\newblock {\em Amer. J. Math.}, 106:489--504, 1984.

\bibitem[BNP01]{BNP01}
Christopher Bendel, Daniel Nakano, and Cornelius Pillen.
\newblock On comparing the cohomology of algebraic groups, finite {C}hevalley
  groups and {F}robenius kernels.
\newblock {\em J. Pure Appl. Algebra}, 163(2):119--146, 2001.

\bibitem[BNP02]{BNP02}
Christopher Bendel, Daniel Nakano, and Cornelius Pillen.
\newblock Extensions for finite {C}hevalley groups. {II}.
\newblock {\em Trans. Amer. Math. Soc.}, 354(11):4421--4454 (electronic), 2002.

\bibitem[BNP04a]{BNP04}
Christopher Bendel, Daniel Nakano, and Cornelius Pillen.
\newblock Extensions for finite {C}hevalley groups. {I}.
\newblock {\em Adv. Math.}, 183(2):380--408, 2004.

\bibitem[BNP04b]{BNP04-Frob}
Christopher Bendel, Daniel Nakano, and Cornelius Pillen.
\newblock Extensions for {F}robenius kernels.
\newblock {\em J. Algebra}, 272(2):476--511, 2004.

\bibitem[BNP06]{BNP06}
Christopher Bendel, Daniel Nakano, and Cornelius Pillen.
\newblock Extensions for finite groups of {L}ie type. {II}. {F}iltering the
  truncated induction functor.
\newblock In {\em Representations of algebraic groups, quantum groups, and
  {L}ie algebras}, volume 413 of {\em Contemp. Math.}, pages 1--23. Amer. Math.
  Soc., Providence, RI, 2006.

\bibitem[BNP11]{BNP11}
Christopher Bendel, Daniel Nakano, and Cornelius Pilllen.
\newblock On the vanishing ranges for the cohomology of finite groups of {L}ie
  type.
\newblock {\em Int. Math. Res. Not.}, 2011.

\bibitem[Bou82]{Bourb82}
Nicolas Bourbaki.
\newblock {\em \'{E}l{\'e}ments de math{\'e}matique: groupes et alg{\`e}bres de
  {L}ie}.
\newblock Masson, Paris, 1982.
\newblock Chapitre 9. Groupes de Lie r{{\'e}}els compacts. [Chapter 9. Compact
  real Lie groups].

\bibitem[CPS75]{CPS75}
Ed~Cline, Brian Parshall, and Leonard Scott.
\newblock Cohomology of finite groups of {L}ie type, i.
\newblock {\em Publ. Math. I.H.E.S.}, 45:169--191, 1975.

\bibitem[CPS77]{CPS77}
Ed~Cline, Brian Parshall, and Leonard Scott.
\newblock Induced modules and affine quotients.
\newblock {\em Math. Ann.}, 230(1):1--14, 1977.

\bibitem[CPS83a]{Dect}
Ed~Cline, Brian Parshall, and Leonard Scott.
\newblock Detecting rational cohomology.
\newblock {\em J. London Math. Soc.}, 28:293--300, 1983.

\bibitem[CPS83b]{CPS83}
Ed~Cline, Brian Parshall, and Leonard Scott.
\newblock A {M}ackey imprimitivity theory for algebraic groups.
\newblock {\em Math. Z.}, 182(4):447--471, 1983.

\bibitem[CPS93]{CPS93}
Ed~Cline, Brian Parshall, and Leonard Scott.
\newblock Abstract {K}azhdan-{L}usztig theories.
\newblock {\em Tohoku Math. J.}, 45:511--534, 1993.

\bibitem[CPS09]{CPS09}
Ed~Cline, Brian Parshall, and Leonard Scott.
\newblock Reduced standard modules and cohomology.
\newblock {\em Trans. Amer. Math. Soc.}, 361(10):5223--5261, 2009.

\bibitem[CPSvdK77]{CPSK77}
Edward Cline, Brian Parshall, Leonard Scott, and Wilberd van~der Kallen.
\newblock Rational and generic cohomology.
\newblock {\em Invent. Math.}, 39(2):143--163, 1977.

\bibitem[FP83]{FP83}
Eric Friedlander and Brian Parshall.
\newblock On the cohomology of algebraic and related finite groups.
\newblock {\em Inventiones Math.}, 74(1):85--117, 1983.

\bibitem[FP86]{FP86}
Eric Friedlander and Brian Parshall.
\newblock Cohomology of infinitesimal and discrete groups.
\newblock {\em Math. Ann.}, 273(3):353--374, 1986.

\bibitem[Jan03]{Jantzen}
Jens Jantzen.
\newblock {\em Representations of algebraic groups}, volume 107 of {\em
  Mathematical Surveys and Monographs}.
\newblock American Mathematical Society, Providence, RI, second edition, 2003.

\bibitem[Kop84]{Kop84}
Markku Koppinen.
\newblock Good bimodule filtrations for coordinate rings.
\newblock {\em J. London Math. Soc.}, 30(2):244--250, 1984.

\bibitem[PS11]{PS11}
Brian Parshall and Leonard Scott.
\newblock Bounding {E}xt for modules for algebraic groups, finite groups and
  quantum groups.
\newblock {\em Adv. Math.}, 226(3):2065--2088, 2011.

\bibitem[Sin94]{Sin94}
Peter Sin.
\newblock Extensions of simple modules for special algebraic groups.
\newblock {\em J. Algebra}, 170(3):1011--1034, 1994.

\bibitem[Stea]{SteSL3}
David Stewart.
\newblock The second cohomology of simple ${SL}_3$-modules.
\newblock {\em Comm. Algebra (to appear), arxiv:0907.4626v2}.

\bibitem[Steb]{SteUnb}
David Stewart.
\newblock Unbounding {E}xt.
\newblock {\em J. Algebra (to appear), arxiv:1101.3004}.

\bibitem[UGA11]{Georgia}
UGA Algebra~{VIGRE} Group, \emph{Second cohomology for finite groups of Lie
  type}, arXiv1110.0228v2 (2011).

\end{thebibliography}
\end{document}